\documentclass[11pt, reqno]{amsart}
\usepackage{enumitem}
\makeatletter
\newcommand{\mylabel}[2]{#2\def\@currentlabel{#2}\label{#1}}
\makeatother
\usepackage{array}
\usepackage[english]{babel}
\usepackage[top=1in, bottom=1in, left=1in, right=1in]{geometry}
\expandafter\let\csname ver@amsthm.sty\endcsname\relax

\usepackage{tikz}

\usepackage{amsmath}
\usepackage{amssymb}
\usepackage{mathdots}
\usepackage{mathtools}

\usepackage{dsfont}
\usepackage{chngpage}

\usepackage{hyperref}
\usepackage{amsthm}
\usepackage[capitalize,noabbrev]{cleveref}

\allowdisplaybreaks

\numberwithin{equation}{section}

\newtheorem{thm}{Theorem}[section]
\newtheorem{lemma}[thm]{Lemma}

\newtheorem{prop}[thm]{Proposition}

\newtheorem{Example}[thm]{Example}

\newtheorem{Remark}[thm]{Remark}

\crefname{thm}{Theorem}{Theorems}
\crefname{lemma}{Lemma}{Lemmas}
\crefname{cor}{Corollary}{Corollaries}
\crefname{prop}{Proposition}{Propositions}

\crefname{example}{Example}{Examples}
\crefname{remark}{Remark}{Remarks}

\newcommand{\emailhref}[1]{\email{\href{#1}{#1}}}

\newcommand{\RR}{\mathbb{R}}
\newcommand{\RRgt}{\RR_{>0}}
\newcommand{\Zgt}{{\mathbb{Z}}_{\geq 1}}
\newcommand{\Zgtwo}{{\mathbb{Z}}_{\geq 2}}
\newcommand{\Zge}{{\mathbb{Z}}_{\ge 0}}

\newcommand{\f}[1]{f_{m,a}(#1)} 

\title{On the maximum of the weighted binomial sum $(1+a)^{-r}\sum_{i=0}^{r}\binom{m}{i}a^{i}$}

\author[Seok Hyun Byun]{Seok Hyun Byun}\emailhref{sbyun@clemson.edu}
\address{School of Mathematical and Statistical Sciences, Clemson University, Clemson, SC, U.S.A.}
\thanks{S.B. is supported by the AMS-Simons Travel Grant.}

\author[Svetlana Poznanovi\'c]{Svetlana Poznanovi\'c}\emailhref{spoznan@clemson.edu}
\address{School of Mathematical and Statistical Sciences, Clemson University, Clemson, SC, U.S.A.}
\thanks{S.P. was supported by NSF DMS 1815832.}

\begin{document}
\maketitle

\begin{abstract}
Recently, Glasby and Paseman considered the following sequence of binomial sums $\{2^{-r}\sum_{i=0}^{r}\binom{m}{i}\}_{r=0}
^{m}$ and showed that this sequence is unimodal and attains its maximum value at $r=\lfloor\frac{m}{3}\rfloor+1$ for $m\in\Zge\setminus\{0,3,6,9,12\}$. They also analyzed the asymptotic behavior of the maximum value of the sequence as $m$ approaches infinity. In the present work, we generalize their results by considering the sequence $\{(1+a)^{-r}\sum_{i=0}^{r}\binom{m}{i}a^{i}\}_{r=0}
^{m}$ for integers $a \geq 1$. We also consider a family of discrete probability distributions that naturally arises from this sequence.
\end{abstract}

\section{Introduction}

A sequence $(a_0,a_1,\dots,a_d)$ of real numbers is \textit{unimodal} if $a_0 \leq a_1 \leq \dots \leq a_{j-1} \leq a_j \geq a_{j+1}\geq \dots \geq a_d$, for some $j$. Unimodal sequences occur frequently in mathematics. For example, the $n$-th row of Pascal's triangle $\binom{n}{0},\binom{n}{1},\dots, \binom{n}{n}$ is unimodal as is its generalization $\binom{n}{0}_{q},\binom{n}{1}_{q},\dots, \binom{n}{n}_{q}$ of $q$-binomial coefficients for $q \geq 0$. However, proving unimodality is often difficult and various tools, including algebra, analysis, probability, and combinatorics, have been used to establish the unimodality of a given sequence. We refer to~\cite{stanley1989log, brenti1989unimodal, brenti1994update,MR3409348} for a broad overview of the subject. If a unimodal sequence is not symmetric, then natural questions are whether the peak is unique and where it is located. These questions are not easy to answer in general.

The motivation for the present work is a recent result of Glasby and Paseman~\cite{glasby2022maximum}, who considered the following weighted binomial sum\begin{equation}\label{Glasby}
    f_{m}(r)\coloneqq\frac{1}{2^r}\sum_{i=0}^{r}\binom{m}{i}.
\end{equation}
In their work, they show that the sequence $\{f_{m}(r)\}_{r=0}^m$ is unimodal with a unique peak at $r = \lfloor \frac{m}{3}\rfloor +1$ for $m\in\Zge\setminus\{0,3,6,9,12\}$. 
Their motivation for determining these properties comes from the fact that the sequence appears in coding theory and information theory \cite{robert1990ash}. Namely, it is desirable for a linear code to have a large rate (to communicate a lot of information) and a large minimal distance (to correct many errors). Thus, for a linear code with parameters $[n, k, d]$, it is desirable that both $k/n$ and $d/n$ are large. The case when $kd/n$ is large was studied in~\cite{alizadeh2022sequences}. A Reed-Muller code $\mathrm{RM}(r, m)$ has $n = 2^m$, $k =\sum_{i=0}^{r}\binom{m}{i}$ and $d = 2^{m-r}$~\cite{ling2004coding}, and hence $kd/n = f_{m}(r)$. 

We consider a natural generalization of the above sequence. Let $a \in \mathbb{R}_{> 0}$ and let
\begin{equation*}
    \f{r}\coloneqq \frac{1}{(1+a)^r}\sum_{i=0}^{r}\binom{m}{i}a^i.
\end{equation*}
Clearly, $f_{m,1}(r) = f_m(r)$. Moreover, $\{\f{r}\}_{r=0}^{m}$ is a nonnegative sequence whose initial and last terms are 1, just like~\eqref{Glasby}. In fact, when $m \in \{0,1\}$, all the terms are 1. So, we focus on $m \in \Zgtwo$. Let 
\begin{equation} \label{radef} r_a(m) \coloneqq \Big\lfloor\frac{a(m+1)+2}{2a+1}\Big\rfloor = \Big\lfloor\frac{am-(a-1)}{2a+1}\Big\rfloor+1.\end{equation} The value $r_a$ depends on $m$, but since $m$ is fixed most of the time, we suppress that dependence in the notation and we write $r_a$ whenever there is no confusion. As we will show, the sequence $\{\f{r}\}_{r=0}^{m}$ is unimodal and if $a \in \Zgt$, the unique peak is located at $r = r_{a}$ (for most of the cases) or at $r = r_{a} -1$ (for some exceptions). Our main result is the following.

\begin{thm}\label{1.1} Let $a \in \mathbb{R}_{>0}$ and $m \in \Zgtwo$. Then
\begin{enumerate}[label=(\alph*)]
    \item[\mylabel{1.1.(a)}{(a)}] the sequence $\{\f{r}\}_{r=0}^{m}$ is log-concave and thus unimodal.
\end{enumerate}
\noindent Moreover, if $a \in \Zgt$, then \begin{enumerate}[label=(\alph*)]
    \item[\mylabel{1.1.(b)}{(b)}]  for all $m\in\Zgtwo \setminus\{3,2a+4, 4a+5\}$ (except the case when $a=1$, where $m$ also cannot be $12 = 6a+6$), the sequence $\{\f{r}\}_{r=0}^{m}$ attains its unique maximum value at $r=r_{a} = \big\lfloor\frac{am-(a-1)}{2a+1}\big\rfloor+1$. For $m\in\{3,2a+4, 4a+5\}$ or $(a,m)=(1,12)$, the sequence attains its unique maximum value at $r=r_{a}-1=\big\lfloor\frac{am-(a-1)}{2a+1}\big\rfloor$.  
    \item[\mylabel{1.1.(c)}{(c)}]
    \begin{equation*}
        \lim_{m\to \infty} f_{m,r}(r_{a})\sqrt{m}\bigg(\frac{1+a}{1+2a}\bigg)^{m}= \frac{(1+a)^{\frac{1}{2}}(1+2a)}{a^{\frac{3}{2}}\sqrt{2\pi}},
    \end{equation*}
    and thus
    \begin{equation*}
        \lim_{a\to \infty}\lim_{m\to \infty} f_{m,r}(r_{a})\sqrt{m}\bigg(\frac{1+a}{1+2a}\bigg)^{m}= \sqrt{\frac{2}{\pi}}.
    \end{equation*}
\end{enumerate}
\end{thm}

We prove the log-concavity of $\{\f{r}\}_{r=0}^{m}$ in Section~\ref{sec:log-concav}. Section~\ref{sec:peak} contains the main ingredients for the proof about the location of the peak in the nonexceptional cases, while Section~\ref{sec:long prop} contains the proof of a key lemma used in Section~\ref{sec:peak}. The proof of Theorem~\ref{1.1}~\ref{1.1.(b)} is completed in Section~\ref{exceptions}. We find bounds for the maximal value of $\{\f{r}\}_{r=0}^{m}$ and prove its asymptotic behavior in Section~\ref{sec:maxvalue}. In Section~\ref{sec:avgvalue}, we consider the probability distribution obtained by appropriately normalizing the sequence $\{\f{r}\}_{r=0}^{m}$. Interestingly, we find that the mean of this distribution is asymptotically close to the mode $r_a$. Finally, we end with some comments in Section~\ref{sec:discussion}. In many places, we use the general ideas from~\cite{glasby2022maximum}, but we have to come up with different arguments because some of the proofs in~\cite{glasby2022maximum} rely on the fact that $a=1$. We give more detail on how our approach compares to the approach in~\cite{glasby2022maximum} in each section.


\section{Log-concavity of $\{\f{r}\}_{r=0}^{m}$} \label{sec:log-concav}

Recall that a sequence of nonnegative real numbers $\{x_{i}\}_{i\geq0}$ is \textit{log-concave} if $x_{k}^{2}\geq x_{k-1}x_{k+1}$ for all $k\geq1$, and it is \textit{strongly log concave} if $x_{k}x_{l}\geq x_{k-1}x_{l+1}$ for all $0<k\leq l$. In fact, these two notions are equivalent. It is also well-known that log-concave sequences of positive numbers are unimodal. Also, as explained in \cite{brenti1989unimodal}, term-wise (Hadamard) product of nonnegative log-concave finite sequences is log-concave. The following simple lemma states that the sequence of partial sums of Hadamard product of two nonnegative log-concave sequences is also log-concave. Although likely known, we could not find this result in the literature, so we state it here with proof for completeness. We use it to show that our sequence $\{\f{r}\}_{r=0}^{m}$ is log-concave, and thus it is unimodal.

\begin{lemma}\label{Lemma2.1}
    Let $\{x_{i}\}_{i\geq0}$ and $\{y_{i}\}_{i\geq0}$ be two sequences of nonnegative real numbers that are log-concave. For each nonnegative integer $k$, we set $z_{k}\coloneqq\sum_{i=0}^{k}x_{i}y_{i}$. Then, the sequence $\{z_{i}\}_{i\geq0}$ is log-concave.
\end{lemma}

\begin{proof}
Let $\{u_{i}\}$ be a log-concave sequence of nonnegative real numbers. The log-concavity of $\{\sum_{i=0}^{k}u_{i}\}_{k\geq0}$ is equivalent to

\begin{equation*}        
    \Bigg(\sum_{i=0}^{k}u_{i}\Bigg)^{2}\geq \Bigg(\sum_{i=0}^{k-1}u_{i}\Bigg)\Bigg(\sum_{i=0}^{k+1}u_{i}\Bigg).
\end{equation*}
If we subtract $\displaystyle\Bigg(\sum_{i=0}^{k-1}u_{i}\Bigg)\Bigg(\sum_{i=0}^{k}u_{i}\Bigg)$ from both sides of the above inequality, we have

\begin{equation*}       
    \Bigg(\sum_{i=0}^{k}u_{i}\Bigg)\cdot u_{k}\geq \Bigg(\sum_{i=0}^{k-1}u_{i}\Bigg)\cdot u_{k+1}.
\end{equation*}
Since log-concave sequences are strongly log-concave, by the log-concavity of $\{u_{i}\}_{i\geq0}$, we have $u_{i+1}u_{k}\geq u_{i}u_{k+1}$ for all $0\leq i\leq k-1$. Hence,

\begin{equation*}
    \Bigg(\sum_{i=0}^{k}u_{i}\Bigg)\cdot u_{k}
    \geq\Bigg(\sum_{i=1}^{k}u_{i}\Bigg)\cdot u_{k}
    =\sum_{i=0}^{k-1}u_{i+1}\cdot u_{k}
    \geq\sum_{i=0}^{k-1}u_{i}\cdot u_{k+1}\\
    =\Bigg(\sum_{i=0}^{k-1}u_{i}\Bigg)\cdot u_{k+1}.
\end{equation*}
Setting $u_{k}=x_{k}y_{k}$ proves the result.
\end{proof}

\begin{proof}[Proof of Theorem \ref{1.1}~\ref{1.1.(a)}]
To check the log-concavity of the sequence $\{f_{m,a}(r)\}_{r=0}^{m}$, we need to check $f_{m,a}^{2}(r) \geq f_{m,a}(r-1)f_{m,a}(r+1) \text{ for all } 1\leq r<m$.
This is equivalent to 
\begin{equation*}
    \Bigg(\sum_{i=0}^{r}\binom{m}{i}a^{i}\Bigg)^{2} \geq \Bigg(\sum_{i=0}^{r-1}\binom{m}{i}a^{i}\Bigg)\Bigg(\sum_{i=0}^{r+1}\binom{m}{i}a^{i}\Bigg).
\end{equation*}
Note that the last inequality is asserting that the sequence $\{\sum_{i=0}^{k}x_{i}y_{i}\}_{k=0}^{m}$ is log-concave for $x_{i}=\binom{m}{i}$ and $y_{i}=a^{i}$. This follows from Lemma~\ref{Lemma2.1} and the well-known facts that $\{\binom{m}{i}\}_{i=0}^{m}$ and $\{a^{i}\}_{i=0}^{m}$ are log-concave.
\end{proof}


\section{Peak location for $\{\f{r}\}_{r=0}^{m}$} \label{sec:peak}

In this section, we show that the sequence $\{\f{r}\}_{r=0}^{m}$ achieves its maximum value at $r=r_a$ for all but finitely many integers $m$. We prove it by showing the following two propositions.

\begin{prop}\label{prop:firstinequality}
For $a \in \mathbb{Z}_{\geq 1}$ and $m\in\Zgtwo\setminus\{3,2a+4,4a+5\}$ (except the case when $a=1$, where $m$ also cannot be $12 = 6a+6$),
\begin{equation*}
    \f{r_{a}-1}<\f{r_{a}},  
\end{equation*}
where $r_{a}$ is given by~\eqref{radef}.
\end{prop}

\begin{prop}\label{prop:secondinequality}
For $a \in \mathbb{Z}_{\geq 1}$ and $m\in\Zgtwo$ ,
\begin{equation*}
    \f{r_{a}}>\f{r_{a}+1},
\end{equation*}
where $r_{a}$ is given by~\eqref{radef}.
\end{prop}

\subsection {Proof of Proposition~\ref{prop:firstinequality}}

First, we note that

\begin{equation}\label{eq:prop1_eq}
    f_{m,a}(r_{a}-1)<f_{m,a}(r_{a}) \iff \sum_{i=0}^{r_{a}-1}\binom{m}{i}a^{i} < \binom{m}{r_{a}}a^{r_{a}-1}.   
\end{equation}
Before we proceed, we make the following observation:
\begin{equation*}
    \frac{am-(a-1)}{2a+1} \text{ is an integer} \iff am\equiv a-1 \text{ (mod $2a+1$)}.
\end{equation*}
Since $a$ and $2a+1$ are relatively prime, the above linear congruence has a unique solution and one can check that $m\equiv 3 \text{ (mod $2a+1$)}$ is the unique solution.

We split the proof of the right-hand inequality in~\eqref{eq:prop1_eq} into two cases: when $\frac{am-(a-1)}{2a+1}$ is not an integer (i.e., when $m\not\equiv 3 \text{ (mod $2a+1$)}$) and when $\frac{am-(a-1)}{2a+1}$ is an integer (i.e., $m\equiv 3 \text{ (mod $2a+1$)}$). The first case follows from Lemma~\ref{lem:case1_lem2} below, and we had to find a proof that uses a  different idea from~\cite{glasby2022maximum} because the original proof for $a=1$ in \cite{glasby2022maximum} cannot be applied to other positive integer values of $a$. The proof of the second case follows from Lemma~\ref{lem:case2_lem1} below, which is a generalization of a result in~\cite{glasby2022maximum}. The following lemma, which we prove first, is used in the proof of  Lemma~\ref{lem:case1_lem2}.

\begin{lemma}{\label{lem:case1_lem1}} Let $a \in \Zgt$ and $m \geq 4$ such that $m\not\equiv 3 \text{ (mod $2a+1$)}$. For any integer $i$ such that $1\leq i< r_{a}$, we have
\begin{equation*}
    \binom{m}{i-1}a^{i-1}+\binom{m}{i}a^{i}\leq \binom{m}{i+1}a^{i}-\binom{m}{i-1}a^{i-2}.
\end{equation*}
\end{lemma}

\begin{proof}

The above inequality can be simplified to the equivalent inequality as follows:

\begin{equation}
\begin{aligned} \label{lemma3.3ineq1}
    &\binom{m}{i-1}a^{i-1}+\binom{m}{i}a^{i}\leq \binom{m}{i+1}a^{i}-\binom{m}{i-1}a^{i-2}\\
    &\iff\binom{m}{i-1}a^{i-1}+\binom{m}{i-1}a^{i-2}\leq \binom{m}{i+1}a^{i}-\binom{m}{i}a^{i}\\
    &\iff\binom{m}{i-1}a^{i-2}(a+1)\leq\bigg(\frac{m-i}{i+1}-1\bigg)\binom{m}{i}a^{i}\\
    &\iff i(i+1)(a+1)\leq (m-2i-1)(m-i+1)a^2.
\end{aligned}
\end{equation}
Since $i<r_{a}$ and $\frac{am-(a-1)}{2a+1}$ is not an integer (this is because $m\not\equiv 3 \text{ (mod $2a+1$)}$), we have 
\begin{equation} \label{something} i\leq r_{a}-1 = \bigg\lfloor\frac{am-(a-1)}{2a+1}\bigg\rfloor \leq \frac{am-a}{2a+1}.
\end{equation}
Hence, $m-2i-1\geq m-\frac{2(am-a)}{2a+1}-1=\frac{(2a+1)m-2(am-a)-(2a+1)}{2a+1}=\frac{m-1}{2a+1}>0$ and similarly $m-i+1>0$.
Therefore, the last inequality in~\eqref{lemma3.3ineq1} is equivalent to
\begin{equation*}
 \frac{i}{m-2i-1}\cdot\frac{i+1}{m-i+1} \leq \frac{a^2}{a+1}.
\end{equation*}
   We claim that for any $1\leq i< r_{a}$, we have
\begin{equation}\label{partial}
    \frac{i}{m-2i-1}\leq a \text{   and   } \frac{i+1}{m-i+1}\leq \frac{a}{a+1}.
\end{equation}
One can readily see that
\begin{equation*}
    \frac{i}{m-2i-1}\leq a \iff i\leq \frac{am-a}{2a+1}
\end{equation*}
and
\begin{equation*}
    \frac{i+1}{m-i+1}\leq \frac{a}{a+1} \iff i\leq \frac{am-1}{2a+1}.
\end{equation*}
The above two inequalities follow from ~\eqref{something} and the fact that $a\geq1$. 
\end{proof}

\begin{lemma}{\label{lem:case1_lem2}} Let $a \in \Zgt$ and $m \geq 4$ s.t. $m\not\equiv 3 \text{ (mod $2a+1$)}$. Then $\sum_{i=0}^{r_{a}-1}\binom{m}{i}a^{i} < \binom{m}{r_{a}}a^{r_{a}-1}$.
\end{lemma}

\begin{proof}
Since $m\geq4$, we have $r_{a}\geq2$. The proof of the claim is slightly different, depending on the parity of $r_{a}$.

\noindent \emph{(a) The case when $r_{a}$ is even.}

In this case, by Lemma~\ref{lem:case1_lem1},

\begin{align*}
    \sum_{i=0}^{r_{a}-1}\binom{m}{i}a^{i}=\sum_{j=0}^{\frac{r_{a}-2}{2}}\Bigg[\binom{m}{2j}a^{2j}+\binom{m}{2j+1}a^{2j+1}\Bigg]
    &\leq\sum_{j=0}^{\frac{r_{a}-2}{2}}\Bigg[\binom{m}{2j+2}a^{2j+1}-\binom{m}{2j}a^{2j-1}\Bigg]\\
    &=\binom{m}{r_{a}}a^{r_{a}-1}-\binom{m}{0}a^{-1}\\
    &<\binom{m}{r_{a}}a^{r_{a}-1}.
\end{align*}

\noindent \emph{(b) The case when $r_{a}$ is odd.}

Again, by Lemma~\ref{lem:case1_lem1},

\begin{align*}
    \sum_{i=0}^{r_{a}-1}\binom{m}{i}a^{i}&=\binom{m}{0}+\sum_{j=1}^{\frac{r_{a}-1}{2}}\Bigg[\binom{m}{2j-1}a^{2j-1}+\binom{m}{2j}a^{2j}\Bigg]\\
    &\leq\binom{m}{0}+\sum_{j=1}^{\frac{r_{a}-1}{2}}\Bigg[\binom{m}{2j+1}a^{2j}-\binom{m}{2j-1}a^{2j-2}\Bigg]\\
    &=\binom{m}{0}+\binom{m}{r_{a}}a^{r_{a}-1}-\binom{m}{1}\\
    &<\binom{m}{r_{a}}a^{r_{a}-1}.
\end{align*}   
\end{proof}

\begin{lemma}{\label{lem:case2_lem1}}
For any integers $a\geq 2$ and $k\geq3$ (and also for integers $a=1$ and $k\geq4$),
\begin{equation}\label{eq:lemma3.5_ineq}
    \sum_{i=0}^{ak+1}\binom{(2a+1)k+3}{i}a^{i} < \binom{(2a+1)k+3}{ak+2}a^{ak+1}.
\end{equation}
\end{lemma}

\begin{proof}
    The above inequality can be rewritten as follows:
\begin{equation*}
    \sum_{i=0}^{ak+1}\binom{(2a+1)k+3}{ak+1-i}a^{ak+1-i} < \binom{(2a+1)k+3}{ak+2}a^{ak+1}.
\end{equation*}
If we divide both sides of the above inequality by $\binom{(2a+1)k+3}{ak+1}a^{ak+1}$, then we have
\begin{equation*}
    \sum_{i=0}^{ak+1}\frac{(ak+2-1)\cdots(ak+2-i)}{((a+1)k+2+1)\cdots((a+1)k+2+i)}a^{-i} < \frac{(a+1)k+2}{ak+2},
\end{equation*}
which is equivalent to
\begin{equation*}
    \sum_{i=0}^{ak+1}\Bigg(\frac{1}{a}\cdot\frac{ak+2-1}{(a+1)k+2+1}\Bigg)\cdots\Bigg(\frac{1}{a}\cdot\frac{ak+2-i}{(a+1)k+2+i}\Bigg) < a^{-1}\cdot\Bigg[\frac{a}{1}\cdot\frac{(a+1)k+2}{ak+2}\Bigg].
\end{equation*}
If we set $X_{i}=\frac{1}{a}\cdot\frac{ak+2-i}{(a+1)k+2+i}$ for nonnegative integers $i \in [0, ak+1]$, then we can rewrite the above inequality as follows:
\begin{equation*}
    \sum_{i=0}^{ak+1}X_{1}X_{2}\cdots X_{i} < a^{-1}\cdot X_{0}^{-1}.
\end{equation*}
Observe that for any integer $i$ such that $4< i\leq ak+1$, we have
\begin{equation*}
    X_{i}< X_{4}=\frac{1}{a}\cdot\frac{ak-2}{(a+1)k+6}<\frac{1}{a+1}.
\end{equation*}
Thus,

\begin{align*}
    \sum_{i=0}^{ak+1}X_{1}X_{2}\cdots X_{i}&=1+X_{1}+X_{1}X_{2}+X_{1}X_{2}X_{3}+X_{1}X_{2}X_{3}X_{4}\sum_{i=4}^{ak+1}X_{5}\cdots X_{i}\\
    &<1+X_{1}+X_{1}X_{2}+X_{1}X_{2}X_{3}+X_{1}X_{2}X_{3}X_{4}\sum_{i=4}^{ak+1}\bigg(\frac{1}{a+1}\bigg)^{i-4}\\
    &<1+X_{1}+X_{1}X_{2}+X_{1}X_{2}X_{3}+X_{1}X_{2}X_{3}X_{4}\sum_{i=4}^{\infty}\bigg(\frac{1}{a+1}\bigg)^{i-4}\\
    &=1+X_{1}+X_{1}X_{2}+X_{1}X_{2}X_{3}+X_{1}X_{2}X_{3}X_{4}\cdot\frac{a+1}{a}.
\end{align*}
Thus, if we show the following inequality
\begin{align*}
    &1+X_{1}+X_{1}X_{2}+X_{1}X_{2}X_{3}+X_{1}X_{2}X_{3}X_{4}\cdot\frac{a+1}{a}<a^{-1}\cdot X_{0}^{-1} \\
    &\iff a^{-1}\cdot X_{0}^{-1}-\Big(1+X_{1}+X_{1}X_{2}+X_{1}X_{2}X_{3}+X_{1}X_{2}X_{3}X_{4}\cdot\frac{a+1}{a}\Big)>0,
\end{align*}
then we are done.
If we replace $X_{i}$ by $X_{i}=\frac{1}{a}\cdot\frac{ak+2-i}{(a+1)k+2+i}$ in the above inequality and clear all denominators by multiplying with $a^4(ak+2)\prod_{i=3}^{6}((a+1)k+i)$, we get

\begin{align*}
    &a^6(k^4-2k^3)+a^5(5k^4+3k^3-30k^2)+a^4(10k^4+20k^3-52k^2-148k)\\
    &+a^3(2k^3-44k^2-208k-240)+a^2(5k^3+2k^2+12k)+a(-4k)+(-4k)>0.
\end{align*}

When $a=1$, the left-hand side of the last inequality becomes $16k^4+28k^3-124k^2-352k-240$ and one can check that it is positive for all integers $k\geq 4$ (and it is negative for $k=3$). This proves the claim when $a = 1$.

Assume henceforth that $a \geq 2$ and $k \geq 3$. We estimate the quadratic in $a$ as follows:
\begin{align*}
    a^2(5k^3+2k^2+12k)+a(-4k)+(-4k)&>a^2(5k^3+2k^2+12k)+a^2(-4k)+a^2(-4k)\\
    &=a^2(5k^3+2k^2+4k)\\
    &>0.    
\end{align*}
Hence, it is enough to show that the remaining polynomial in $a$ satisfies
\begin{align*}
    &a^6(k^4-2k^3)+a^5(5k^4+3k^3-30k^2)+a^4(10k^4+20k^3-52k^2-148k)\\
    &+a^3(2k^3-44k^2-208k-240)>0.
\end{align*}
By dividing both sides by $a^3>0$, we see that this is equivalent to
\begin{align*}
    &a^3(k^4-2k^3)+a^2(5k^4+3k^3-30k^2)+a(10k^4+20k^3-52k^2-148k)\\
    &+(2k^3-44k^2-208k-240)>0.
\end{align*}
One can readily check that for  $k\geq3$, the terms $k^4-2k^3$, $5k^4+3k^3-30k^2$, and $10k^4+20k^3-52k^2-148k$ are positive. Hence, using that $a \geq 2$, we get
\begin{align*}
    &a^3(k^4-2k^3)+a^2(5k^4+3k^3-30k^2)+a(10k^4+20k^3-52k^2-148k)\\
    &+(2k^3-44k^2-208k-240)\\
    &\geq8(k^4-2k^3)+4(5k^4+3k^3-30k^2)+2(10k^4+20k^3-52k^2-148k)\\
    &+(2k^3-44k^2-208k-240)\\
    &=48k^4+38k^3-268k^2-504k-240.
\end{align*}
Finally, one can check that $48k^4+38k^3-268k^2-504k-240$ is positive for all integers $k\geq3$.

\end{proof}
\begin{proof}[Proof of Proposition~\ref{prop:firstinequality}] 
As described after~\eqref{eq:prop1_eq}, we split the proof into two cases.

\noindent \underline{\textit{Case 1:}} $\frac{am-(a-1)}{2a+1}$ is not an integer.

In this case, we need to show that the inequality $f_{m,a}(r_{a}-1)<f_{m,a}(r_{a})$ holds for every $m\geq2$ such that $m\not\equiv3 \text{ (mod }2a+1)$. When $m=2$, we have $r_a = 1$, and this case can be easily checked. Clearly $m \neq 3$, and the case when $m\geq4$ follows from Lemma~\ref{lem:case1_lem2} and~\eqref{eq:prop1_eq}.

\noindent \underline{\textit{Case 2:}} $\frac{am-(a-1)}{2a+1}$ is an integer.

In this case, since $m\equiv3 \text{ (mod }2a+1)$, we can write $m=(2a+1)k+3$ for some $k \in \Zge$ and we have $r_{a}=\big\lfloor\frac{a\{(2a+1)k+3\}-(a-1)}{2a+1}\big\rfloor+1=ak+2$. The right-hand side  inequality in~\eqref{eq:prop1_eq} is equivalent to~\eqref{eq:lemma3.5_ineq}. If $k =0, 1, 2, 3$, then
$m = 3, 2a + 4, 4a + 5, 6a + 6$, respectively. The first three cases are excluded by the hypothesis of Proposition~\ref{prop:firstinequality}, and $m = 6a + 6$ is excluded when $a = 1$. The result now
follows from Lemma~\ref{lem:case2_lem1} and~\eqref{eq:prop1_eq}.
\end{proof}

\subsection {Proof of Proposition~\ref{prop:secondinequality}}
We start by rewriting the inequality we need to prove. Namely, note that 
\begin{equation}\label{ineq:equivalent}
    f_{m,a}(r_{a})> f_{m,a}(r_{a}+1) \iff \sum_{i=0}^{r_{a}}\binom{m}{i}a^{i} > \binom{m}{r_{a}+1}a^{r_{a}}.
\end{equation}

To prove this inequality, we need three lemmas. The first lemma states that $f_{m,a}(r)> f_{m,a}(r+1)$ implies $f_{m-1,a}(r)> f_{m-1,a}(r+1)$.

\begin{lemma}
\label{lem:firstineqforsecondprop}
    For $r\in\Zge$, $a\in \RRgt$, and $m\in\Zgt$ such that $r<m$,
    \begin{equation}\label{ineq:Lemma3.6}
\sum_{i=0}^{r}\binom{m}{i}a^i\geq\binom{m}{r+1}a^{r} \implies \sum_{i=0}^{r}\binom{m-1}{i}a^i>\binom{m-1}{r+1}a^{r}.
    \end{equation}
\end{lemma}

\begin{proof}
    By definition of the binomial coefficient, we have
    \begin{equation*}
        \binom{m-1}{i}=\frac{m-i}{m}\binom{m}{i}
    \end{equation*}
    for any integer $i$ such that $0\leq i\leq r$. Thus, if $\sum_{i=0}^{r}\binom{m}{i}a^i\geq\binom{m}{r+1}a^{r}$, then we have

    \begin{align*}
   \sum_{i=0}^{r}\binom{m-1}{i}a^i =\sum_{i=0}^{r}\frac{m-i}{m}\binom{m}{i}a^i > \sum_{i=0}^{r}\frac{m-r-1}{m}\binom{m}{i}a^i &=\frac{m-r-1}{m}\sum_{i=0}^{r}\binom{m}{i}a^i\\
   &\geq\frac{m-r-1}{m}\binom{m}{r+1}a^{r}\\
    &=\binom{m-1}{r+1}a^{r}.
    \end{align*}

\end{proof}
The second lemma states that for sufficiently small $r$, the inequality $f_{m,a}(r)> f_{m,a}(r+1)$ implies $f_{m+2,a}(r+1)> f_{m+2,a}(r+2)$. This is needed when we deal with the case $a > 1$.

\begin{lemma}
\label{lem:secondineqforsecondprop}
    For $r \in \Zge$, $a \in \RRgt$, and $m \in \Zgt$ such that $r\leq\lfloor\frac{m-1}{2}\rfloor$,
    \begin{equation}\label{eq:lemma3.7}
     \sum_{i=0}^{r}\binom{m}{i}a^i \geq \binom{m}{r+1}a^r \implies \sum_{i=0}^{r+1}\binom{m+2}{i}a^i > \binom{m+2}{r+2}a^{r+1}.
    \end{equation}
    \end{lemma}
\begin{proof} The quadratic $i \mapsto  (m-i+1)(i+1)$ has zeroes at $i =-1$ and $i = m+1$, and hence has a maximum at the midpoint $\frac{-1+(m+1)}{2} = \frac{m}{2}$. So, for $-1 \leq i \leq r \leq \frac{m-1}{2}$, we have
$(m-i+1)(i+1) \leq (m-r+1)(r+1)$. However, $(m-r+1)(r+1) \leq (m-r)(r+2)$ since $2r+1 \leq m$. Therefore, for any integers $i,m,$ and $r$ such that $0\leq i\leq r\leq\lfloor\frac{m-1}{2}\rfloor$, we have
    \begin{equation*}
        \binom{m+2}{i+1}=\frac{(m+2)(m+1)}{(m-i+1)(i+1)}\binom{m}{i}\geq\frac{(m+2)(m+1)}{(m-r)(r+2)}\binom{m}{i}.
    \end{equation*}
    Using this inequality, if $\sum_{i=0}^{r}\binom{m}{i}a^i \geq \binom{m}{r+1}a^r$, then we have

    \begin{align*}
        \sum_{i=0}^{r+1}\binom{m+2}{i}a^i &>\sum_{i=1}^{r+1}\binom{m+2}{i}a^i =\sum_{i=0}^{r}\binom{m+2}{i+1}a^{i+1}
        \geq\frac{(m+2)(m+1)}{(m-r)(r+2)}\sum_{i=0}^{r}\binom{m}{i}a^{i+1} \\
        &\geq\frac{(m+2)(m+1)}{(m-r)(r+2)}\binom{m}{r+1}a^{r+1}
        =\binom{m+2}{r+2}a^{r+1}.
    \end{align*}
\end{proof}

The following lemma will serve as a base case of our inductive argument in the proof of Proposition \ref{prop:secondinequality}.

\begin{lemma} \label{lem:longlemma}
    For $a \in \Zgt$ and $k \in \Zge$,
    \begin{equation}\label{ineq:longlemma}
        \sum_{i=0}^{ak+2}\binom{(2a+1)k+5}{i}a^i>\binom{(2a+1)k+5}{ak+3}a^{ak+2}.
    \end{equation}
\end{lemma}

The proof of Lemma~\ref{lem:longlemma}  takes several steps, and we postpone it till the next section. Here, we show how the last three lemmas imply Proposition~\ref{prop:secondinequality}.

\begin{proof}[Proof of Proposition~\ref{prop:secondinequality} subject to Lemma~\ref{lem:longlemma}] Recall that we rephrased Proposition~\ref{prop:secondinequality} as the inequality~\eqref{ineq:equivalent}. When $m=2$, we have $r_{a}=1$ and the inequality~\eqref{ineq:equivalent}  can be easily checked. 

Now we verify~\eqref{ineq:equivalent} for $m\geq3$. The set $\Zge$ can be partitioned into classes parameterized by the quotient when dividing by $2a+1$. More precisely, if we let $J_{k}:=\{(2a+1)k+l\;|\;0\leq l\leq2a\}$ for nonnegative integers $k$, then $\Zge=\bigsqcup_{k\geq0}J_{k}$.

If we set $I_{k}:=J_{k}+\{3\}=\{j + 3 \;|\; j \in J_k\} = \{(2a+1)k+l\;|\;3\leq l\leq2a+3\}$, then the the collection $\{I_{k}\}_{k\geq0}$ forms a partition of $\mathbb{Z}_{\geq3}$, i.e., $\mathbb{Z}_{\geq3} = \bigsqcup_{k\geq0}I_{k}$. We use a combination of applications of Lemma~\ref{lem:firstineqforsecondprop} and Lemma~\ref{lem:secondineqforsecondprop} to show that~\eqref{ineq:equivalent} holds for each element of $I_k$, $k\geq 0$. See Figure~\ref{fig:Implication} for an illustration of how these lemmas are applied.

\begin{figure}[h!]
    \centering
    \includegraphics[width=1\textwidth]{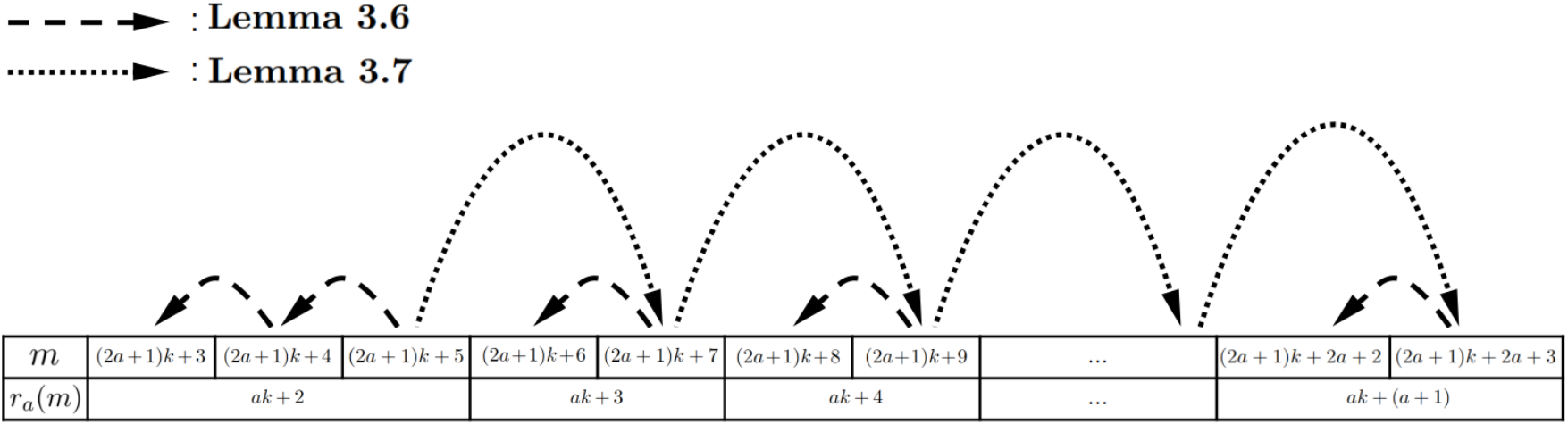}
    \caption{An illustration of the logic of the proof of Proposition~\ref{prop:secondinequality}.}
    \label{fig:Implication}
\end{figure}

In order to apply the two lemmas, we make some observations about the values of $r_{a}$ for each element  $m \in I_{k}$. Suppose that $m = (2a + 1)k + l$. For $l=3$, we have $m = (2a+1)k+3$ and
\begin{equation*}
r_a(m) = \Big\lfloor \frac{am-(a-1)}{2a+1}\Big\rfloor + 1 =\lfloor ak+1\rfloor+1=ak+2.
\end{equation*}
On the other hand, for each $3<l\leq 2a+3$, $m = (2a+1)k+l$ and we have
\begin{equation*}
am - (a-1) = (2a+1)(ak-1+ \Big\lfloor \frac{l}{2} \Big\rfloor) + R,
\end{equation*}
where 
\begin{equation*}
    R=\begin{cases}
			(a+2)-\frac{l}{2}, & \text{if $l$ is even}\\
            (2a+2)-\frac{l-1}{2}, & \text{if $l$ is odd.}
		 \end{cases}
\end{equation*}
If $l$ is even, then $l \geq 4$, and \[0\leq\Big\lfloor \frac{R}{2a+1}\Big\rfloor = \Big\lfloor \frac{(a+2)-\frac{l}{2}} {2a+1} \Big\rfloor \leq  \Big\lfloor \frac{(a+2)-2} {2a+1} \Big\rfloor  = 0.\]
If $l$ is odd, then $l \geq 5$, and \[0\leq\Big\lfloor \frac{R}{2a+1}\Big\rfloor = \Big\lfloor \frac{(2a+2)-\frac{l-1}{2}} {2a+1} \Big\rfloor \leq  \Big\lfloor \frac{(2a+2)-2} {2a+1} \Big\rfloor  = 0.\]
Thus, $\lfloor \frac{R}{2a+1}\rfloor=0$ and
\begin{equation*}
r_a(m) = \Big\lfloor \frac{am-(a-1)}{2a+1}\Big\rfloor + 1 = (ak-1+ \Big\lfloor \frac{l}{2} \Big\rfloor) + \Big\lfloor \frac{R}{2a+1}\Big\rfloor +1 = ak + \Big\lfloor \frac{l}{2} \Big\rfloor. 
\end{equation*}
The values $r_a(m)$ are illustrated in Figure~\ref{fig:Implication}. In particular, for $l\in\{5,7,\dots,2a+1\}$, one can see that \[ r_a(m) = ak+\Big\lfloor\frac{l}{2}\Big\rfloor = ak+\Big\lfloor\frac{l-1}{2}\Big\rfloor \leq ak+\Big\lfloor\frac{k+l-1}{2}\Big\rfloor=\Big\lfloor\frac{(2a+1)k+l-1}{2}\Big\rfloor = \Big\lfloor\frac{m-1}{2}\Big\rfloor,\]
$r_{a}(m+2)=r_{a}(m)+1$, and $r_{a}(m)=r_{a}(m-1)$.

Note that $|I_k| = 2a + 1 \geq 3$. Lemma~\ref{lem:longlemma} shows that~\eqref{ineq:equivalent} holds for the third element in each $I_{k}$,  namely for $m=(2a+1)k+5 \in I_{k}$. Then, an inductive argument using Lemma~\ref{lem:secondineqforsecondprop} shows that~\eqref{ineq:equivalent} holds for \[m \in \{(2a+1)k+5, (2a+1)k+7, \dots,(2a+1)k+2a+3\} \subseteq I_k.\]  Applying Lemma~\ref{lem:firstineqforsecondprop} for each of these values of $m$, noting that $r_a(m) = r_a(m-1)$, we conclude that~\eqref{ineq:equivalent} holds for \[m \in \{(2a+1)k+4, (2a+1)k+6, \dots,(2a+1)k+2a+2\} \subseteq I_k.\] Finally, one more application of Lemma~\ref{lem:firstineqforsecondprop} for $m = (2a+1)k+4$ implies that~\eqref{ineq:equivalent} holds with $m = (2a+1)k+3$.
\end{proof}


\section{Proof of Lemma~\ref{lem:longlemma}} \label{sec:long prop}

To prove Lemma~\ref{lem:longlemma}, we introduce two families of polynomials $P_{i}(a,k)$ and $Q_{i}(a,k)$ as follows:

\begin{itemize}
    \item $P_{0}(a,k)=Q_{0}(a,k)\coloneqq1$.
    \item $Q_{n+1}(a,k)\coloneqq (ak+3-n)Q_{n}(a,k)$ for $n \in \mathbb{Z}_{\geq 0}$.
    \item $P_{n+1}(a,k)\coloneqq a((a+1)k+3+n)P_{n}(a,k)-aQ_{n+1}(a,k)$ for $n \in \mathbb{Z}_{\geq 0}$.
\end{itemize}
The polynomials $P_{n}(a,k)$ and $Q_{n}(a,k)$ for $i=0,1,\dots,6$ are shown in Table~\ref{table1}.
\begin{table} [h!]
\begin{tabular}
{ | m{0.3cm} | m{10cm}| m{5cm} | } 
  \hline
  $n$ & $P_{n}(a,k)$ & $Q_{n}(a,k)$ \\ 
  \hline
  $0$ & $1$ & $1$ \\ 
  \hline
  $1$ & $ak$ & $ak+3$ \\
  \hline
  $2$ & $a^2k^2-a^2k-6a$ & $a^2k^2+5ak+6$ \\ 
  \hline
  $3$ & $a^3k^3-(a^4+2a^3)k^2-(11a^3+17a^2)k-(30a^2+6a)$ & $a^3k^3+6a^2k^2+11ak+6$ \\ 
  \hline
  $4$ & $a^4k^4-(a^6+3a^5+2a^4)k^3-(17a^5+40a^4+28a^3)k^2-(96a^4+138a^3+12a^2)k-(180a^3+36a^2)$ & $a^4k^4+6a^3k^3+11a^2k^2+6ak$ \\ 
  \hline
  $5$ & $a^5k^5-(a^8+4a^7+5a^6)k^4-(24a^7+78a^6+82a^5+33a^4)k^3-(215a^6+514a^5+346a^4+7a^3)k^2-(852a^5+1182a^4+120a^3-6a^2)k-(1260a^4+252a^3)$ & $a^5k^5+5a^4k^4+5a^3k^3-5a^2k^2-6ak$ \\ 
  \hline
  $6$ & $a^6k^6 - (a^{10} + 5a^9 + 9a^8 + 5a^7 - 5a^6)k^5 - (32a^9 + 134a^8 + 200a^7 + 115a^6 + 28a^5)k^4 - (407a^8 + 1353a^7 + 1516a^6 + 617a^5 - 8a^4)k^3 - (2572a^7 + 6146a^6 + 4070a^5 + 170a^4 - 2a^3)k^2 - (8076a^6 + 10968a^5 + 1212a^4 - 48a^3 + 12a^2)k - (10080a^5 + 2016a^4)
$ & $a^6k^6+3a^5k^5-5a^4k^4-15a^3k^3+4a^2k^2+12ak$ \\ 
  \hline
\end{tabular}
    \caption{The first few polynomials $P_{n}(a,k)$ and $Q_{n}(a,k)$.}
    \label{table1}
\end{table}
Our motivation for introducing the polynomials $P_{n}(a,k)$ and $Q_{n}(a,k)$ is the following fact.
\begin{prop}\label{prop:equivalent}
For $a \in \Zgt$, $k \in \Zge$, and an integer $n$ such that $0\leq n\leq ak+2$, the inequality \eqref{ineq:longlemma} from Lemma~\ref{lem:longlemma} is equivalent to the following inequality:
    \begin{equation}\label{ineq:goal3}
        \sum_{i=0}^{ak+2-n}\binom{(2a+1)k+5}{i}a^i>\frac{P_{n}(a,k)}{Q_{n}(a,k)}\binom{(2a+1)k+5}{ak+3-n}a^{ak+2-n}.
    \end{equation}
\end{prop}
\begin{proof}
Adding the inequality~\eqref{ineq:longlemma} to the equality~\eqref{eq:equivalent} below gives the inequality~\eqref{ineq:goal3} above. Hence to show the inequality~\eqref{ineq:longlemma} is equivalent to the inequality~\eqref{ineq:goal3}, it is enough to show that
\begin{equation}\label{eq:equivalent}
     \displaystyle\binom{(2a+1)k+5}{ak+3}a^{ak+2}-\sum_{i=ak+2-(n-1)}^{ak+2}\binom{(2a+1)k+5}{i}a^i=\frac{P_{n}(a,k)}{Q_{n}(a,k)}\binom{(2a+1)k+5}{ak+3-n}a^{ak+2-n}
\end{equation}
holds for any integer $0\leq n\leq ak+2$. We give an inductive proof of \eqref{eq:equivalent}. It is clearly true when $n=0$ because the sum on the left-hand side of \eqref{eq:equivalent} is empty and $P_{0}(a,k)=Q_{0}(a,k)=1$. Suppose \eqref{eq:equivalent} holds for some nonnegative integer $n<ak+2$. Then, using the induction hypothesis and the recurrences for $P_{n+1}(a,k)$ and $Q_{n+1}(a,k)$, we have
    \begin{align*}
        \displaystyle&\binom{(2a+1)k+5}{ak+3}a^{ak+2}-\sum_{i=ak+2-n}^{ak+2}\binom{(2a+1)k+5}{i}a^i\\
        &=\frac{P_{n}(a,k)}{Q_{n}(a,k)}\binom{(2a+1)k+5}{ak+3-n}a^{ak+2-n}-\binom{(2a+1)k+5}{ak+2-n}a^{ak+2-n}\\
        &=\frac{P_{n}(a,k)}{Q_{n}(a,k)}\frac{(a+1)k+3+n}{ak+3-n}\binom{(2a+1)k+5}{ak+2-n}a^{ak+2-n}-\binom{(2a+1)k+5}{ak+2-n}a^{ak+2-n}\\
        &=\frac{P_{n}(a,k)((a+1)k+3+n)-Q_{n}(a,k)(ak+3-n)}{Q_{n}(a,k)(ak+3-n)}\binom{(2a+1)k+5}{ak+2-n}a^{ak+2-n}\\
        &=\frac{a((a+1)k+3+n)P_{n}(a,k)-aQ_{n+1}(a,k)}{Q_{n+1}(a,k)}\binom{(2a+1)k+5}{ak+2-n}a^{ak+2-(n+1)}\\
        &=\frac{P_{n+1}(a,k)}{Q_{n+1}(a,k)}\binom{(2a+1)k+5}{ak+3-(n+1)}a^{ak+2-(n+1)}.
    \end{align*}
Hence, the claim holds for $n+1$.

\end{proof}

$P_{n}(a,k)$ and $Q_{n}(a,k)$ are polynomials in two variables $a$ and $k$, but we regard them as polynomials in a variable $k$ with coefficients in $\mathbb{Z}[a]$. One can easily check that for every nonnegative integer $n$, the polynomials $P_{n}(a,k)$ and $Q_{n}(a,k)$ have degree $n$ with leading coefficients $a^n$. We write these polynomials in descending order of exponents in a variable $k$ as follows:

\begin{equation*}
    P_{n}(a,k)=a^{n}k^{n}-\sum_{i=0}^{n-1}p_{n,i}(a)k^{i} \text{   and   } Q_{n}(a,k)=a^{n}k^{n}+\sum_{i=0}^{n-1}q_{n,i}(a)k^{i}.
\end{equation*}

As can be seen from Table~\ref{table1}, for $i \in \{2,3,4,5,6\}$, all coefficients of $k$ in $P_{i}(a,k)$  but the leading one are negative for $a \in \mathbb{R}_{\geq1}$. We prove this property is true for all $i \geq 2$ and then use this to prove Lemma~\ref{lem:longlemma}. We note that Glasby and Paseman~\cite{glasby2022maximum} also considered the quotient $P_{i}(a,k)/Q_{i}(a,k)$ for $a=1$, but in that case $P_{i}(a,k)$ and $Q_{i}(a,k)$ have common factors and they were able to completely characterize the signs of the coefficients in the numerator and the denominator after reduction. Such reduction is not possible for general $a$, and we had to find different bounds for the coefficients $p_{n,i}(a)$.

We now turn our focus on the properties of the coefficients $p_{n,i}(a)$ and $q_{n,i}(a)$. As a consequence of the definition of the polynomial $Q_{n+1}(a,k)$, we have $\displaystyle Q_{n+1}(a,k)=\prod_{i=0}^{n}(ak+3-i)$. Thus,
\begin{equation*}
    q_{n+1,0}(a)=\prod_{i=0}^{n}(3-i)
\end{equation*}
and
\begin{equation}\label{formula:qnplusoneq}
    q_{n+1,n}(a)=\sum_{i=0}^{n}(3-i)a^{n}=\frac{(6-n)(n+1)}{2}a^{n}.
\end{equation}
Furthermore, if $n\geq 3$,
\begin{equation*}
    q_{n+1,1}(a)=a\cdot\prod_{0\leq i\leq n, i\neq3}(3-i)=6a\cdot(-1)^{n-3}\cdot(n-3)!.
\end{equation*}
Also, since  $Q_{n+1}(a,k) = (ak+3-n)Q_{n}(a,k)$,   we have that for $1 \leq j \leq n-1$,
\begin{equation}\label{eq:recurrdegQ}
    q_{n+1,j}(a)=aq_{n,j-1}(a)+(3-n)q_{n,j}(a).
\end{equation}
Similarly, using the defining recurrence $P_{n+1}(a,k)=a((a+1)k+3+n)P_{n}(a,k)-aQ_{n+1}(a,k)$ and $\displaystyle q_{n+1,0}(a)=\prod_{i=0}^{n}(3-i)$, we have $p_{0,0}(a)=1$,
 $p_{1,0}(a)=0$, $p_{2,0}(a)=6a$, $p_{3,0}(a)=30a^{2}+6a$, and
\begin{equation*}
    p_{n+1,0}(a)=a(n+3)p_{n,0}(a) \text{ for } n\geq 3.
\end{equation*}
Thus, for any $n\geq2$,
\begin{equation*}
    p_{n+1,0}(a)=a^{n-2}\frac{(n+3)!}{5!}(30a^2+6a).
\end{equation*}
Using the above recurrence for $P_{n+1}(a,k)$ and $q_{n+1,n}(a)=\frac{(6-n)(n+1)}{2}a^{n}$, we have
\begin{equation}\label{eq:Pnplusonen}
    p_{n+1,n}(a)=a(a+1)p_{n,n-1}(a)-\frac{n(n-3)}{2}a^{n+1}.
\end{equation}
Lastly, by comparing the coefficients of $k^{j}$ for $1\leq j\leq n-1$ in the defining recurrence for $P_{n+1}(a,k)$, we have
\begin{equation}
\begin{aligned}\label{eq:longrecurr}
    p_{n+1,j}(a)&=a(a+1)p_{n,j-1}(a)+a(3+n)p_{n,j}(a)+aq_{n+1,j}(a)\\
    &=a(a+1)p_{n,j-1}(a)+a(3+n)p_{n,j}(a)+a^2q_{n,j-1}(a)+a(3-n)q_{n,j}(a).
\end{aligned}
\end{equation}

Using the above equations, we prove the following two propositions that give us lower bounds for the coefficients of $P_{n}(a,k)$.

\begin{prop}\label{lem:Pnnminusone}
    For any real number $a\geq1$ and an integer $n\geq3$,
    \begin{equation}\label{formula:Pnnminusone}
        p_{n,n-1}(a)\geq a^{2n-2}+(n-1)a^{2n-3}+\frac{n(n-3)}{2}a^{2n-4}.
    \end{equation}
\end{prop}

\begin{proof}
    We use induction on $n$. When $n=3$, $p_{3,2}(a)=a^{4}+2a^{3}$, hence the claim holds. Suppose the claim is true for all $a\geq1$ and an integer $n\geq3$. Then, by the recurrence \eqref{eq:Pnplusonen} and the induction hypothesis, we have
    \begin{align*}
        p_{n+1,n}(a)&=a(a+1)p_{n,n-1}(a)-\frac{n(n-3)}{2}a^{n+1}\\
        &\geq (a^{2}+a)(a^{2n-2}+(n-1)a^{2n-3}+\frac{n(n-3)}{2}a^{2n-4})-\frac{n(n-3)}{2}a^{n+1}\\
        &=a^{2n}+na^{2n-1}+\frac{(n+1)(n-2)}{2}a^{2n-2}+\frac{n(n-3)}{2}(a^{n-4}-1)a^{n+1}\\
        &\geq a^{2n}+na^{2n-1}+\frac{(n+1)(n-2)}{2}a^{2n-2}.
    \end{align*}
Hence, the proposition also holds for $n+1$.
\end{proof}

\begin{prop}\label{lem:pnipositive}
    For any real number $a\geq1$ and integers $n\geq3$ and $i$ such that $1\leq i\leq n-1$,
    \begin{equation} \label{crazy}
        p_{n,i-1}(a)>|q_{n,i-1}(a)|+(n-3)|q_{n,i}(a)|.
    \end{equation}
Consequently, $p_{n,i-1}(a)>0$.
\end{prop}

\begin{proof}
    Again, we use mathematical induction on $n$ to prove the claim. When $n=3,4,5,6$, the 14 inequalities~\eqref{crazy} are given in Table~\ref{table2} and one can check that the proposition holds directly. 
    \begin{table} [h]
\begin{tabular} 
{ | m{0.3cm} | m{0.3cm} | m{6.5cm}|| m{0.3cm} | m{0.3cm}| m{6.5cm}| } 
  \hline
  $n$ & $i$ & & $n$ & $i$ &\\
  \hline
  3 & 1 & $30a^2+6a > 6$ & 5 & 3 & $215a^6+514a^5+346a^4+7a^3 > 10a^3 + 5a^2$ \\ 
  \hline
  3 & 2 & $11a^3+17a^2 > 11a$ & 5 & 4 & $24a^7+78a^6+82a^5+33a^4 > 10a^4 + 5a^3$ \\ 
  \hline
  4 & 1 & $180a^3+36a^2 > 6a$ & 6 & 1 & $10080a^5 + 2016a^4 > 36a$ \\
  \hline
  4 & 2 & $96a^4+138a^3+12a^2 > 11a^2 + 6a$ & 6 & 2 & $8076a^6 + 10968a^5 + 1212a^4 - 48a^3 + 12a^2 > 12a^2 + 12a$\\ 
  \hline
  4 & 3 & $17a^5+40a^4+28a^3 > 6a^3 + 11a^2$ & 6 & 3 &$2572a^7 + 6146a^6 + 4070a^5 + 170a^4 - 2a^3 > 45a^3 + 4a^2$\\
  \hline
  5 & 1 & $1260a^4+252a^3 > 12a$& 6 & 4 & $407a^8 + 1353a^7 + 1516a^6 + 617a^5 - 8a^4 > 15a^4 + 15a^3$\\
  \hline
  5 & 2 & $852a^5+1182a^4+120a^3-6a^2 > 10a^2 + 6a$& 6 & 5 & $32a^9 + 134a^8 + 200a^7 + 115a^6 + 28a^5 > 9a^5 + 5a^4$\\
  \hline
\end{tabular}
    \caption{The inequality~\eqref{crazy} for the first few values of $n$.}
    \label{table2}
\end{table}

Suppose the claim is true for some $n\geq6$. This assumption together with Proposition \ref{lem:Pnnminusone} implies $p_{n,i-1}(a)>0$  for all $a\geq1$ and $1\leq i\leq n$. We will show that
    \begin{equation}\label{rev2}
        p_{n+1,i-1}(a)>|q_{n+1,i-1}(a)|+(n-2)|q_{n+1,i}(a)|
    \end{equation}
     for all  $a\geq1$ and integers $1\leq i\leq n$ by considering three cases ($2\leq i\leq n-1$, $i=n$, and $i=1$) separately. 

    For $2\leq i\leq n-1$, by the recurrence \eqref{eq:longrecurr},
    \begin{align*}
        p_{n+1,i-1}(a)&=a(a+1)p_{n,i-2}(a)+a(3+n)p_{n,i-1}(a)+a^2q_{n,i-2}(a)+a(3-n)q_{n,i-1}(a)\\
        &\geq a(a+1)p_{n,i-2}(a)+a(3+n)p_{n,i-1}(a)-a^2|q_{n,i-2}(a)|-a(n-3)|q_{n,i-1}(a)|.
    \end{align*}
From the induction hypothesis, we also have
    \begin{equation*}
        a(a+1)p_{n,i-2}(a)>a(a+1)|q_{n,i-2}(a)|+a(a+1)(n-3)|q_{n,i-1}(a)|
    \end{equation*}
and
    \begin{equation*}
        a(3+n)p_{n,i-1}(a)>a(3+n)|q_{n,i-1}(a)|+a(3+n)(n-3)|q_{n,i}(a)|.
    \end{equation*}
If we combine the last three inequalities and simplify, we have
    \begin{align*}
        p_{n+1,i-1}(a)&> a|q_{n,i-2}(a)|+(a^2(n-3)+a(3+n))|q_{n,i-1}(a)|+a(n+3)(n-3)|q_{n,i}(a)|\\
        &> a|q_{n,i-2}(a)|+((n-3)+a(n-2))|q_{n,i-1}(a)|+(n-2)(n-3)|q_{n,i}(a)|\\
        &= [a|q_{n,i-2}(a)|+(n-3)|q_{n,i-1}(a)|]+(n-2)[a|q_{n,i-1}(a)|+(n-3)|q_{n,i}(a)|]\\
        &= [|aq_{n,i-2}(a)|+|(3-n)q_{n,i-1}(a)|]+(n-2)[|aq_{n,i-1}(a)|+|(3-n)q_{n,i}(a)|]\\
        &\geq |aq_{n,i-2}(a)+(3-n)q_{n,i-1}(a)|+(n-2)|aq_{n,i-1}(a)+(3-n)q_{n,i}(a)|\\
        &=|q_{n+1,i-1}(a)|+(n-2)|q_{n+1,i}(a)|
    \end{align*}
where in the last equality, we use~\eqref{eq:recurrdegQ}.

    For $i=n$, we need to show the following inequality:
    \begin{equation} \label{goal2}
        p_{n+1,n-1}(a)>|q_{n+1,n-1}(a)|+(n-2)|q_{n+1,n}(a)|.
    \end{equation}
    Again, by the recurrence \eqref{eq:longrecurr},
    \begin{align*}
        p_{n+1,n-1}(a)&=a(a+1)p_{n,n-2}(a)+a(3+n)p_{n,n-1}(a)+a^2q_{n,n-2}(a)+a(3-n)q_{n,n-1}(a)\\
        &\geq a(a+1)p_{n,n-2}(a)+a(3+n)p_{n,n-1}(a)-a^2|q_{n,n-2}(a)|-a(n-3)|q_{n,n-1}(a)|.
    \end{align*}
    From the induction hypothesis,
    \begin{equation*}
        a(a+1)p_{n,n-2}(a)>a(a+1)|q_{n,n-2}(a)|+a(a+1)(n-3)|q_{n,n-1}(a)|.
    \end{equation*}
    If we combine the above two inequalities and use~\eqref{eq:recurrdegQ}, then we get
    \begin{align*}
        p_{n+1,n-1}(a)&>a|q_{n,n-2}(a)|+a^{2}(n-3)|q_{n,n-1}(a)|+a(3+n)p_{n,n-1}(a)\\
        &\geq a|q_{n,n-2}(a)|+(n-3)|q_{n,n-1}(a)|+a(3+n)p_{n,n-1}(a)\\
        &=|aq_{n,n-2}(a)|+|(3-n)q_{n,n-1}(a)|+a(3+n)p_{n,n-1}(a)\\
        &\geq|aq_{n,n-2}(a)+(3-n)q_{n,n-1}(a)|+a(3+n)p_{n,n-1}(a)\\
        &=|q_{n+1,n-1}(a)|+a(3+n)p_{n,n-1}(a).
    \end{align*}
    Thus, it suffices to show that
    \begin{equation*}
        a(3+n)p_{n,n-1}(a)\geq(n-2)|q_{n+1,n}(a)|.
    \end{equation*}
    This can be easily verified by using \eqref{formula:qnplusoneq} and \eqref{formula:Pnnminusone} as follows:
    \begin{align*}
        a(3+n)p_{n,n-1}(a)&\geq a(3+n)(a^{2n-2}+(n-1)a^{2n-3}+\frac{n(n-3)}{2}a^{2n-4})\\
        &\geq a(3+n)(a^{n-1}+(n-1)a^{n-1}+\frac{n(n-3)}{2}a^{n-1})\\
        &=(3+n)\cdot\frac{n^2-n}{2}a^{n}\\
        &\geq(3+n)\cdot\frac{n^2-5n-6}{2}a^{n}\\
        &=(3+n)\cdot\frac{(n-6)(n+1)}{2}a^{n}\\
        &\geq(n-2)\cdot\bigg|\frac{(6-n)(n+1)}{2}a^{n}\bigg|\\
        &=(n-2)|q_{n+1,n}(a)|.
    \end{align*}
    
    For $i=1$, we need to show the following inequality:
    \begin{equation} \label{goal3}
        p_{n+1,0}(a)>|q_{n+1,0}(a)|+(n-2)|q_{n+1,1}(a)|.
    \end{equation}
    However, since $q_{n+1,0}(a)=0$ and $\displaystyle q_{n+1,1}(a)=6a\cdot(-1)^{n-3}\cdot(n-3)!$ for $n\geq3$ ,~\eqref{goal3} is equivalent to 
    \begin{equation*}
        p_{n+1,0}(a)>6a\cdot(n-2)!.
    \end{equation*}
This last inequality holds since $p_{n+1,0}(a)=a^{n-2}\frac{(n+3)!}{5!}(30a^2+6a)$ for $n\geq2$, and thus
    \begin{equation*}
        p_{n+1,0}(a)>\frac{(n+3)!}{5!}\cdot6a=\frac{(n+3)!}{5!(n-2)!}\cdot6a(n-2)!=\binom{n+3}{5}\cdot6a(n-2)!\geq 6a(n-2)!.
    \end{equation*}
Hence, the proposition also holds for $n+1$.
\end{proof}

\begin{proof}[Proof of Lemma~\ref{lem:longlemma}]

Using the Proposition~\ref{lem:Pnnminusone} and Proposition~\ref{lem:pnipositive}, we can bound the function $P_{n}(a,k)$ from the above as follows. For any integer $n\geq3$,

\begin{equation}
\begin{aligned}\label{ineq:Pnnegative}
    P_{n}(a,k)=a^{n}k^{n}-\sum_{i=0}^{n-1}p_{n,i}(a)k^{i}&<a^{n}k^{n}-(a^{2n-2}+(n-1)a^{2n-3}+\frac{n(n-3)}{2}a^{2n-4})k^{n-1}\\
    &=[k-(a^{n-2}+(n-1)a^{n-3}+\frac{n(n-3)}{2}a^{n-4})]a^{n}k^{n-1}.
\end{aligned}
\end{equation}
In particular, since $a\geq1$, we have $P_{k+1}(a,k)<0$ for $k\geq2$. On the other hand, $Q_{k+1}(a,k)>0$, since $Q_{n}(a,k) = \prod_{i=0}^{n-1}(ak+3-i)$. Therefore, for $k\geq2$, ~\eqref{ineq:goal3} holds for $n=k+1$ and, by Proposition~\ref{prop:equivalent}, we conclude that~\eqref{ineq:longlemma}  is true for $k\geq2$. The remaining cases when $k=0,1$  in~\eqref{ineq:longlemma} can be checked directly by observing that the values $\frac{P_{2}(a,0)}{Q_{2}(a,0)}$ and  $\frac{P_{2}(a,1)}{Q_{2}(a,1)}$ are negative (these are because $P_{2}(a,0)=P_{2}(a,1)=-6a<0$) and using Proposition~\ref{prop:equivalent}.
\end{proof}


\section{Proof of Theorem $1.1(b)$} \label{exceptions}

Suppose that $a \in \mathbb{Z}_{\geq 1}$. By combining Propositions \ref{prop:firstinequality} and \ref{prop:secondinequality}, we can conclude that for all $m\in\Zgtwo \setminus\{3,2a+4, 4a+5\}$ (except $(a,m)=(1,12)$), the sequence $\{\f{r}\}_{r=0}^{m}$ attains its unique maximum value at $r=r_{a}=\big\lfloor\frac{am-(a-1)}{2a+1}\big\rfloor+1$. To complete the proof of Theorem~\ref{1.1}~\ref{1.1.(b)}, we need to show that for $m\in\{3,2a+4, 4a+5\}$ or $(a,m)=(1,12)$, the sequence attains its unique maximum value at $r=r_{a}-1=\big\lfloor\frac{am-(a-1)}{2a+1}\big\rfloor$. Since we already know that the sequence is unimodal by Section~\ref{sec:log-concav}, it is enough to show that for $a \in \mathbb{Z}_{\geq 1}$ and $m\in\{3,2a+4,4a+5\}$ or when $(a,m)=(1,12)$,
\begin{equation}\label{5.1}
    \f{r_{a}-2}<\f{r_{a}-1}
\end{equation}
and
\begin{equation}\label{5.2}
    \f{r_{a}-1}>\f{r_{a}}
\end{equation}
hold.

We first check the case $(a,m)=(1,12)$. In this case, $r_{a}=\lfloor\frac{1\cdot12-(1-1)}{2+1}\big\rfloor+1=5$. One can see that $\f{r_{a}-2}=\f{3}=\frac{1}{2^3}\sum_{i=0}^{3}\binom{12}{i}1^i=\frac{299}{8}$, 
$\f{r_{a}-1}=\f{4}=\frac{1}{2^4}\sum_{i=0}^{4}\binom{12}{i}1^i=\frac{794}{16}$, 
$\f{r_{a}}=\f{5}=\frac{1}{2^5}\sum_{i=0}^{5}\binom{12}{i}1^i=\frac{1586}{32}$. Since $\frac{299}{8}<\frac{794}{16}$ and $\frac{794}{16}>\frac{1586}{32}$, this proves the case $(a,m)=(1,12)$.

Next, we prove the inequalites \eqref{5.1} and \eqref{5.2} for the case $m = 3$ and $a \in \mathbb{Z}_{\geq 1}$. In this case, $r_{a}=\lfloor\frac{a\cdot3-(a-1)}{2a+1}\big\rfloor+1=\lfloor\frac{2a+1}{2a+1}\big\rfloor+1=2$. Then, $\f{r_{a}-2}=\f{0}=1$, 
$\f{r_{a}-1}=\f{1}=\frac{1}{(a+1)^1}\sum_{i=0}^{1}\binom{3}{i}a^i=\frac{3a+1}{a+1}$, and
$\f{r_{a}}=\f{2}=\frac{1}{(a+1)^2}\sum_{i=0}^{2}\binom{3}{i}a^i=\frac{3a^2+3a+1}{(a+1)^2}$. Again, it is not hard to see that $1<\frac{3a+1}{a+1}$ and $\frac{3a+1}{a+1}>\frac{3a^2+3a+1}{(a+1)^2}$ hold for any $a>0$. This proves the case $m = 3$ and $a \in \mathbb{Z}_{\geq 1}$.

The proofs of the remaining two cases, $m = 2a+4$ and $m = 4a+5$, are less straightforward compared to the previous two cases but are not difficult. When $m = 2a+4$, $r_{a}=\lfloor\frac{a\cdot(2a+4)-(a-1)}{2a+1}\big\rfloor+1=\lfloor\frac{2a^2+3a+1}{2a+1}\big\rfloor+1=a+2$. In this case,~\eqref{5.1} and~\eqref{5.2}, by~\eqref{eq:prop1_eq} and~\eqref{ineq:equivalent}, reduce to
\begin{equation}\label{equival1 for 2a+4}
\f{r_{a}-2}<\f{r_{a}-1} \iff \sum_{i=0}^{a}\binom{2a+4}{i}a^i < \binom{2a+4}{a+1}a^{a}.
\end{equation}
and    
\begin{equation}\label{equival2 for 2a+4}
    \f{r_{a}-1}>\f{r_{a}} \iff \sum_{i=0}^{a+1}\binom{2a+4}{i}a^i > \binom{2a+4}{a+2}a^{a+1}.
\end{equation}
If $a = 1$, then~\eqref{equival1 for 2a+4} holds as $\binom{6}{0} + \binom{6}{1} < \binom{6}{3}$. When $a\geq2$,
\begin{align*}
    \sum_{i=0}^{a}\binom{2a+4}{i}a^i<a^{a}\cdot\sum_{i=0}^{a}\binom{2a+4}{a}a^{i-a}<\Bigg[\sum_{i=0}^{\infty}a^{-i}\Bigg]\binom{2a+4}{a}a^{a}&=\frac{a}{a-1}\binom{2a+4}{a}a^{a}\\
    &\leq\frac{a+4}{a+1}\binom{2a+4}{a}a^{a}\\
    &=\binom{2a+4}{a+1}a^{a}
\end{align*}
Thus, \eqref{equival1 for 2a+4} holds for all $a \in \mathbb{Z}_{\geq 1}$. On the other hand, since $a-1\geq0$,
\begin{align*}
    \sum_{i=0}^{a+1}\binom{2a+4}{i}a^i    &\geq \sum_{i=a-1}^{a+1}\binom{2a+4}{i}a^i\\
    &=\Bigg[\frac{(a+2)(a+1)a}{(a+3)(a+4)(a+5)}\cdot\frac{1}{a^2}+\frac{(a+2)(a+1)}{(a+3)(a+4)}\cdot\frac{1}{a}+\frac{a+2}{a+3}\Bigg]\cdot\binom{2a+4}{a+2}a^{a+1}\\
   &=\frac{a^4+12a^3+47a^2+60a+12}{a^4+12a^3+47a^2+60a}\cdot\binom{2a+4}{a+2}a^{a+1}\\
    &>\binom{2a+4}{a+2}a^{a+1}
\end{align*}
and therefore \eqref{equival2 for 2a+4} holds.
    
Similarly, when $m = 4a+5$, $r_{a}=\lfloor\frac{a\cdot(4a+5)-(a-1)}{2a+1}\big\rfloor+1=\lfloor\frac{4a^2+4a+1}{2a+1}\big\rfloor+1=2a+2$. In this case, \eqref{5.1} and \eqref{5.2}, by~\eqref{eq:prop1_eq} and~\eqref{ineq:equivalent}, reduce to
\begin{equation}\label{equival1 for 4a+5}
    \f{r_{a}-2}<\f{r_{a}-1} \iff \sum_{i=0}^{2a}\binom{4a+5}{i}a^i < \binom{4a+5}{2a+1}a^{2a}
\end{equation}
and    
\begin{equation}\label{equival2 for 4a+5}
    \f{r_{a}-1}>\f{r_{a}} \iff \sum_{i=0}^{2a+1}\binom{4a+5}{i}a^i > \binom{4a+5}{2a+2}a^{2a+1}.
\end{equation}
Showing \eqref{equival1 for 4a+5} holds for $a=1$ and $2$ is again straightforward. When $a\geq3$,
\begin{align*}
    \sum_{i=0}^{2a}\binom{4a+5}{i}a^i< a^{2a}\cdot\sum_{i=0}^{2a}\binom{4a+5}{2a}a^{i-2a}<\Bigg[\sum_{i=0}^{\infty}a^{-i}\Bigg]\binom{4a+5}{2a}a^{2a}
    &=\frac{a}{a-1}\binom{4a+5}{2a}a^{2a}\\
    &<\frac{2a+5}{2a+1}\binom{4a+5}{2a}a^{2a}\\
    &=\binom{4a+5}{2a+1}a^{2a}.
\end{align*}
Thus, \eqref{equival1 for 4a+5} holds for all integers $a \geq 1$. The inequality \eqref{equival2 for 4a+5} can also be checked using a similar idea as the proof of \eqref{equival2 for 2a+4}. Namely, since $2a-2\geq0$,
\begin{align*}
    &\sum_{i=0}^{2a+1}\binom{4a+5}{i}a^i\geq \sum_{i=2a-2}^{2a+1}\binom{4a+5}{i}a^i\\
    &
    \begin{aligned}
    =\Bigg[\frac{(2a+2)(2a+1)2a(2a-1)}{(2a+3)(2a+4)(2a+5)(2a+6)}\cdot\frac{1}{a^3}&+\frac{(2a+2)(2a+1)2a}{(2a+3)(2a+4)(2a+5)}\cdot\frac{1}{a^2}\\
    &+\frac{(2a+2)(2a+1)}{(2a+3)(2a+4)}\cdot\frac{1}{a}+\frac{2a+2}{2a+3}\Bigg]\cdot\binom{4a+5}{2a+2}a^{2a+1}
    \end{aligned}
    \\
    &=\frac{4a^6+38a^5+136a^4+221a^3+140a^2+20a-1}{4a^6+36a^5+119a^4+171a^3+90a^2}\cdot\binom{4a+5}{2a+2}a^{2a+1}\\
    &>\binom{4a+5}{2a+2}a^{2a+1}.
\end{align*}
Therefore, \eqref{equival2 for 4a+5} is verified and Theorem~\ref{1.1}~\ref{1.1.(b)} is proved completely.


\section{Maximal value of $\{\f{r}\}_{r=0}^{m}$} \label{sec:maxvalue}
In this section, we prove Theorem~\ref{1.1}~\ref{1.1.(c)}. We first establish a lower and an upper bound for the maximal value $f_{m,a}(r_{a})$.

\begin{prop} \label{prop:bounds}
    For  $a \in \Zgt$ and an integer $m>6(a+1)$, we have
    \begin{equation*}
        \bigg(1-\frac{(1+2a)r_{a}-am+a+1}{(1+a)(1+r_{a})}\bigg)\bigg(\frac{a}{1+a}\bigg)^{r_{a}-1}\binom{m}{r_{a}} < f_{m,a}(r_{a}) < \bigg(\frac{a}{1+a}\bigg)^{r_{a}-1}\binom{m}{r_{a}},
    \end{equation*}
where $r_{a}$ is given by~\eqref{radef}.

\end{prop}
\begin{proof}
 Let $a \geq 1$. Since, by Theorem~\ref{1.1}~\ref{1.1.(b)} for $m> 6a+6$, $f_{m,a}(r)$ attains its maximum value at $r=r_{a}$, we have
    \begin{equation*}
        f_{m,a}(r_{a}-1) < f_{m,a}(r_{a}) \text{ and } f_{m,a}(r_{a}+1) < f_{m,a}(r_{a}).
    \end{equation*}
    By the definition of $f_{m,a}(r)$ and \eqref{eq:prop1_eq}, we have
    \begin{align*}
        f_{m,a}(r_{a}-1) < f_{m,a}(r_{a})&\iff \sum_{i=0}^{r_{a}-1}\binom{m}{i}a^{i} < \binom{m}{r_{a}}a^{r_{a}-1}\\
        &\iff \sum_{i=0}^{r_{a}}\binom{m}{i}a^{i} < (1+a)a^{r_{a}-1}\binom{m}{r_{a}}\\
        &\iff f_{m,r}(r_{a})=\frac{1}{(1+a)^{r_{a}}}\sum_{i=0}^{r_{a}}\binom{m}{i}a^{i} < \bigg(\frac{a}{1+a}\bigg)^{r_{a}-1}\binom{m}{r_{a}}.
    \end{align*}

    Similarly, using the fact that $\binom{m}{r_{a}+1}=\frac{m-r_{a}}{r_{a}+1}\binom{m}{r_{a}}$ and \eqref{ineq:equivalent}, we have
   \begin{align*}
        f_{m,a}(r_{a}+1) < f_{m,a}(r_{a})&\iff \binom{m}{r_{a}+1}a^{r_{a}} < \sum_{i=0}^{r_{a}}\binom{m}{i}a^{i}\\
        &\iff \bigg(\frac{a(m-r_{a})}{(1+a)(1+r_{a})}\bigg)\bigg(\frac{a}{1+a}\bigg)^{r_{a}-1}\binom{m}{r_{a}} < \frac{1}{(1+a)^{r_{a}}}\sum_{i=0}^{r_{a}}\binom{m}{i}a^{i}\\
        &\iff \bigg(1-\frac{(1+2a)r_{a}-am+a+1}{(1+a)(1+r_{a})}\bigg)\bigg(\frac{a}{1+a}\bigg)^{r_{a}-1}\binom{m}{r_{a}} < f_{m,r}(r_{a}).
    \end{align*}
\end{proof}

\begin{proof}[Proof of Theorem~\ref{1.1}~\ref{1.1.(c)}] 
The definition~\eqref{radef} of $r_{a}$ implies that $2<(1+2a)r_{a}-am+a+1\leq 2a+3$. Thus, we conclude \begin{equation*} \lim_{m\to\infty} \frac{r_a}{m}= \frac{a}{1+2a}, \hspace{1cm} \lim_{m\to\infty}\bigg(1-\frac{(1+2a)r_{a}-am+a+1}{(1+a)(1+r_{a})}\bigg)=1,
\end{equation*}
and, by Proposition~\ref{prop:bounds}, as $m\rightarrow\infty$,\[\f{r_a} \sim \bigg(\frac{a}{1+a}\bigg)^{r_{a}-1}\binom{m}{r_{a}}.\]

Using the Stirling approximation: $n!\sim \sqrt {2\pi n}\left({\frac {n}{e}}\right)^{n}$ as $n \to \infty$, we have
\begin{align*}
    &\displaystyle\bigg(\frac{a}{1+a}\bigg)^{r_{a}-1}\binom{m}{r_{a}}\\
    =&\frac{1+a}{a}\cdot\frac{m!}{r_{a}!(m-r_{a})!}\cdot\Big(\frac{a}{1+a}\Big)^{r_{a}}\\
    =&\frac{1+a}{a}\cdot\bigg(\frac{1+2a}{1+a}\bigg)^{m}\cdot\frac{m!}{(\frac{m}{e})^{m}}\cdot\frac{(\frac{r_{a}}{e})^{r_{a}}}{r_{a}!}\cdot\frac{(\frac{m-r_{a}}{e})^{m-r_{a}}}{(m-r_{a})!}\cdot \bigg(\frac{\frac{am}{2a+1}}{r_{a}}\bigg)^{r_{a}}\cdot\bigg(\frac{\frac{(a+1)m}{2a+1}}{m-r_{a}}\bigg)^{m-r_{a}}\\
    \sim&\frac{1+a}{a}\cdot\bigg(\frac{1+2a}{1+a}\bigg)^{m}\cdot\sqrt{2\pi m}\cdot\frac{1}{\sqrt{2\pi r_{a}}}\cdot\frac{1}{\sqrt{2\pi (m-r_{a})}}\cdot \bigg(\frac{\frac{am}{2a+1}}{r_{a}}\bigg)^{r_{a}}\cdot\bigg(\frac{\frac{(a+1)m}{2a+1}}{m-r_{a}}\bigg)^{m-r_{a}}\\
    =&\frac{1+a}{a}\cdot\bigg(\frac{1+2a}{1+a}\bigg)^{m}\cdot\frac{1}{\sqrt{2\pi m}}\cdot\sqrt{\frac{m}{r_{a}}}\cdot\sqrt{\frac{m}{m-r_{a}}}\cdot \bigg(\frac{\frac{am}{2a+1}}{r_{a}}\bigg)^{r_{a}}\cdot\bigg(\frac{\frac{(a+1)m}{2a+1}}{m-r_{a}}\bigg)^{m-r_{a}}\\
    \sim&\frac{1+a}{a}\cdot\bigg(\frac{1+2a}{1+a}\bigg)^{m}\cdot\frac{1}{\sqrt{2\pi m}}\cdot\sqrt{\frac{1+2a}{a}}\cdot\sqrt{\frac{1+2a}{1+a}}\cdot \bigg(\frac{\frac{am}{2a+1}}{r_{a}}\bigg)^{r_{a}}\cdot\bigg(\frac{\frac{(a+1)m}{2a+1}}{m-r_{a}}\bigg)^{m-r_{a}}\\
    =&\frac{(1+2a)^{m+1}}{a^{\frac{3}{2}}(1+a)^{m-\frac{1}{2}}\sqrt{2\pi m}}\cdot \bigg(\frac{\frac{am}{2a+1}}{r_{a}}\bigg)^{r_{a}}\cdot\bigg(\frac{\frac{(a+1)m}{2a+1}}{m-r_{a}}\bigg)^{m-r_{a}}.
\end{align*}
We now show that 

\begin{equation*}
    \bigg(\frac{\frac{am}{2a+1}}{r_{a}}\bigg)^{r_{a}}\cdot\bigg(\frac{\frac{(a+1)m}{2a+1}}{m-r_{a}}\bigg)^{m-r_{a}}\sim1.
\end{equation*}
This can be verified by considering $(2a+1)$ modulo classes modulo $(2a+1)$. As recorded in Figure \ref{fig:Implication},\\
1) if $m=(2a+1)k+3$, $k\in\Zge$, then $r_{a}=ak+2$ and \\ 
2) if $m=(2a+1)k+2n+\epsilon$, $k\in\Zge$, $n\in\{2,3,\dots,a+1\}$, $\epsilon\in\{0,1\}$, then $r_{a}=ak+n$. 

When $m=(2a+1)k+3$, $k\in\Zge$, as $k \to \infty$,
\begin{align*}
    \bigg(\frac{\frac{am}{2a+1}}{r_{a}}\bigg)^{r_{a}}\cdot\bigg(\frac{\frac{(a+1)m}{2a+1}}{m-r_{a}}\bigg)^{m-r_{a}}&=\bigg(\frac{\frac{a((2a+1)k+3)}{2a+1}}{ak+2}\bigg)^{ak+2}\cdot\bigg(\frac{\frac{(a+1)((2a+1)k+3)}{2a+1}}{(2a+1)k+3-(ak+2))}\bigg)^{(2a+1)k+3-(ak+2)}\\
    &=\bigg(\frac{(2a+1)k+3}{(2a+1)k}\bigg)^{(2a+1)k+3}\bigg(\frac{ak}{ak+2}\bigg)^{ak+2}\bigg(\frac{(a+1)k}{(a+1)k+1}\bigg)^{(a+1)k+1}\\
    &=\bigg(1+\frac{3}{(2a+1)k}\bigg)^{(2a+1)k+3}\bigg(\frac{1}{1+\frac{2}{ak}}\bigg)^{ak+2}\bigg(\frac{1}{1+\frac{1}{(a+1)k}}\bigg)^{(a+1)k+1}\\
    &\sim\bigg(1+\frac{3}{(2a+1)k}\bigg)^{(2a+1)k}\bigg(\frac{1}{1+\frac{2}{ak}}\bigg)^{ak}\bigg(\frac{1}{1+\frac{1}{(a+1)k}}\bigg)^{(a+1)k}\\
    &\sim e^{3}\cdot\frac{1}{e^{2}}\cdot\frac{1}{e^{1}}\\
    &=1.
\end{align*}

Similarly, when $m=(2a+1)k+2n+\epsilon$, $k\in\Zge$, $n\in\{2,3,\dots,a+1\}$, $\epsilon\in\{0,1\}$, as $k \to \infty$,
\begin{align*}
    &\bigg(\frac{\frac{am}{2a+1}}{r_{a}}\bigg)^{r_{a}}\cdot\bigg(\frac{\frac{(a+1)m}{2a+1}}{m-r_{a}}\bigg)^{m-r_{a}}\\
    =&\bigg(\frac{\frac{a((2a+1)k+2n+\epsilon)}{2a+1}}{ak+n}\bigg)^{ak+n}\cdot\bigg(\frac{\frac{(a+1)((2a+1)k+2n+\epsilon)}{2a+1}}{(2a+1)k+2n+\epsilon-(ak+n))}\bigg)^{(2a+1)k+2n+\epsilon-(ak+n)}\\
    =&\bigg(\frac{(2a+1)k+2n+\epsilon}{(2a+1)k}\bigg)^{(2a+1)k+2n+\epsilon}\bigg(\frac{ak}{ak+n}\bigg)^{ak+n}\bigg(\frac{(a+1)k}{(a+1)k+n+\epsilon}\bigg)^{(a+1)k+n+\epsilon}\\
    =&\bigg(1+\frac{2n+\epsilon}{(2a+1)k}\bigg)^{(2a+1)k+2n+\epsilon}\bigg(\frac{1}{1+\frac{n}{ak}}\bigg)^{ak+n}\bigg(\frac{1}{1+\frac{n+\epsilon}{(a+1)k}}\bigg)^{(a+1)k+n+\epsilon}\\
    \sim&\bigg(1+\frac{2n+\epsilon}{(2a+1)k}\bigg)^{(2a+1)k}\bigg(\frac{1}{1+\frac{n}{ak}}\bigg)^{ak}\bigg(\frac{1}{1+\frac{n+\epsilon}{(a+1)k}}\bigg)^{(a+1)k}\\
    \sim& e^{2n+\epsilon}\cdot\frac{1}{e^{n}}\cdot\frac{1}{e^{n+\epsilon}}\\
    =&1.
\end{align*}

Therefore, we have
\begin{equation*}
    \displaystyle\bigg(\frac{a}{1+a}\bigg)^{r_{a}-1}\binom{m}{r_{a}}\sim \frac{(1+2a)^{m+1}}{a^{\frac{3}{2}}(1+a)^{m-\frac{1}{2}}\sqrt{2\pi m}}.
\end{equation*}
and we can conclude that

\begin{equation*}
        \lim_{m\to \infty} f_{m,r}(r_{a})\sqrt{m}\bigg(\frac{1+a}{1+2a}\bigg)^{m}=\lim_{m\to \infty} \bigg(\frac{a}{1+a}\bigg)^{r_{a}-1}\binom{m}{r_{a}}\sqrt{m}\bigg(\frac{1+a}{1+2a}\bigg)^{m}= \frac{(1+a)^{\frac{1}{2}}(1+2a)}{a^{\frac{3}{2}}\sqrt{2\pi}}.
    \end{equation*}
The second equality of part~\ref{1.1.(c)} can be easily obtained by taking $\lim_{a\rightarrow\infty}$ to the first equality.
\end{proof}


\section{The weighted average of $\{\f{r}\}_{r=0}^{m}$} \label{sec:avgvalue}

Since the sequence $\{f_{m,a}(r)\}_{r=0}^{m}$ consists of positive terms, it is natural to consider a discrete probability distribution obtained by normalizing the sequence. More precisely, let $S_{m,a}=\sum_{r=0}^{m}f_{m,a}(r)$ and let  $p_{m,a}(r)\coloneqq f_{m,a}(r)/S_{m,a}$ for $r  \in \{0, 1, \dots, m\}$. We consider the random variable  $X_{m,a}$ whose probability mass function is $\mathrm{P}(X_{m,a}=r) = p_{m,a}(r)$, for $0 \leq r \leq m$.

Given a probability distribution, one natural question is finding its moments. In this section, we find the first moment (i.e., the average) of $X_{m,a}$.

\begin{thm} For $a\in\mathbb{R}_{>0}$ and $m,r\in\mathbb{Z}_{\geq0}$ such that $r \leq m$, let $\mathrm{P}(X_{m,a}=r) = p_{m,a}(r)$. Then
\begin{equation*}
    \mathbb{E}(X_{m,a})\sim \frac{am}{2a+1}+\frac{1}{a} \quad \text{ as } \quad a\rightarrow\infty. 
\end{equation*}
\end{thm}

\begin{proof}
First, we compute the normalizing factor.
\begin{align*}
    S_{m,a}=\sum_{r=0}^{m}f_{m,a}(r)&=\sum_{r=0}^{m}\frac{1}{(1+a)^r}\sum_{i=0}^{r}\binom{m}{i}a^{i}\\
    &=\sum_{i=0}^{m}\Bigg[\sum_{r=i}^{m}\frac{1}{(1+a)^r}\Bigg]\binom{m}{i}a^{i}\\
    &=\sum_{i=0}^{m}\frac{\frac{1}{(a+1)^{i}}\left(1-\frac{1}{(a+1)^{m+1-i}}\right)}{1-\frac{1}{a+1}}\binom{m}{i}a^{i}\\
    &=\frac{a+1}{a}\sum_{i=0}^{m}\Bigg[\frac{1}{(a+1)^{i}}-\frac{1}{(a+1)^{m+1}}\Bigg]\binom{m}{i}a^{i}\\
    &=\frac{a+1}{a}\sum_{i=0}^{m}\binom{m}{i}\bigg(\frac{a}{a+1}\bigg)^{i}-\frac{a+1}{a}\frac{1}{(a+1)^{m+1}}\sum_{i=0}^{m}\binom{m}{i}a^{i}\\
    &=\frac{a+1}{a}\bigg(1+\frac{a}{a+1}\bigg)^{m}-\frac{a+1}{a}\frac{1}{(a+1)^{m+1}}(a+1)^{m}\\
    &=\frac{a+1}{a}\bigg(\frac{2a+1}{a+1}\bigg)^{m}-\frac{1}{a}.
\end{align*}
The following identities are used in the derivations that follow. For nonnegative integer $i\leq m$,
\begin{equation*}
    \sum_{j=0}^{m}j\binom{m}{j}x^j=m(x+1)^{m-1}x,\quad\sum_{r=i}^{m}rx^r=\frac{ix^i-(i-1)x^{i+1}-(m+1)x^{m+1}+mx^{m+2}}{(1-x)^2}.
\end{equation*}
The weighted sum $\sum_{r=0}^{m}r\cdot f_{m,a}(r)$ is

\begin{align*}
    \sum_{r=0}^{m}r\cdot f_{m,a}(r)&=\sum_{r=0}^{m}\frac{r}{(1+a)^r}\sum_{i=0}^{r}\binom{m}{i}a^{i}\\
    &=\sum_{i=0}^{m}\Bigg[\sum_{r=i}^{m}\frac{r}{(1+a)^r}\Bigg]\binom{m}{i}a^{i}\\
    &=\sum_{i=0}^{m}\Bigg[\frac{i(\frac{1}{a+1})^i-(i-1)(\frac{1}{a+1})^{i+1}-(m+1)(\frac{1}{a+1})^{m+1}+m(\frac{1}{a+1})^{m+2}}{(1-\frac{1}{a+1})^2}\Bigg]\binom{m}{i}a^{i}\\
    &
    \begin{aligned}
    =\frac{(a+1)^2}{a^2}\Bigg[&\bigg(1-\frac{1}{a+1}\bigg)\sum_{i=0}^{m}i\binom{m}{i}\bigg(\frac{a}{a+1}\bigg)^i+\frac{1}{a+1}\sum_{i=0}^{m}\binom{m}{i}\bigg(\frac{a}{a+1}\bigg)^i\\
    &-(m+1)\bigg(\frac{1}{a+1}\bigg)^{m+1}\sum_{i=0}^{m}\binom{m}{i}a^i+m\bigg(\frac{1}{a+1}\bigg)^{m+2}\sum_{i=0}^{m}\binom{m}{i}a^i\Bigg]
    \end{aligned}
    \\
    &
    \begin{aligned}
    =\frac{(a+1)^2}{a^2}\Bigg[&\frac{a}{1+a}\cdot m\bigg(1+\frac{a}{a+1}\bigg)^{m-1}\frac{a}{a+1}+\frac{1}{a+1}\bigg(1+\frac{a}{a+1}\bigg)^{m}\\
    &-(m+1)\bigg(\frac{1}{a+1}\bigg)^{m+1}(a+1)^m+m\bigg(\frac{1}{a+1}\bigg)^{m+2}(a+1)^m\Bigg]
    \end{aligned}
    \\
    &=m\bigg(\frac{2a+1}{a+1}\bigg)^{m-1}+\frac{a+1}{a^2}\bigg(\frac{2a+1}{a+1}\bigg)^m-(m+1)\frac{a+1}{a^2}+m\frac{1}{a^2}\\
    &=\Bigg[\frac{a+1}{2a+1}\cdot m+\frac{a+1}{a^2}\Bigg]\Bigg(\frac{2a+1}{a+1}\Bigg)^m-\Bigg(\frac{1}{a}m+\frac{a+1}{a^2}\Bigg).
\end{align*}

Therefore, 

\begin{equation*}
    \mathbb{E}(X_{m,a})=\frac{\big[\frac{a+1}{2a+1}\cdot m+\frac{a+1}{a^2}\big]\big(\frac{2a+1}{a+1}\big)^m-\big(\frac{1}{a}m+\frac{a+1}{a^2}\big)}{\frac{a+1}{a}\big(\frac{2a+1}{a+1}\big)^{m}-\frac{1}{a}}
    \sim \frac{am}{2a+1}+\frac{1}{a}.
\end{equation*}

\end{proof}

Recall that the peak of the sequence $\{f_{m,a}(r)\}_{r=0}^{m}$ was at $r=\lfloor\frac{am-(a-1)}{2a+1}\rfloor+1$. Surprisingly, when $m$ is large enough, the average and the mode $r = r_a$ are so close since $\lim_{m \to \infty} 
\frac{r_a}{m} = \frac{a}{2a+1}$ (the distribution is not symmetric, so this does not need to be the case)!


\section{Discussion}\label{sec:discussion}

\subsection*{\textbf{1.}} We mention that the following strengthening of Lemma~\ref{lem:longlemma} is implied by the properties of the polynomials $P_n(a,k)$ and $Q_n(a,k)$ we proved, but we did not need to use it in the proofs of our results.

\begin{prop}\label{prop:propbeforeproof}
    Let  $a \in \Zgt$, $k \in \Zge$, and let $n \in \mathbb{Z}_{\geq3}$ such that  \[\frac{n-3}{a}\leq k\leq a^{n-2}+(n-1)a^{n-3}+\frac{n(n-3)}{2}a^{n-4}.\] Then
    \begin{equation} \label{secondtolast}
        \sum_{i=ak+2-(n-1)}^{ak+2}\binom{(2a+1)k+5}{i}a^i>\binom{(2a+1)k+5}{ak+3}a^{ak+2}.
    \end{equation}
  \end{prop}
\begin{proof}
   As we checked in \eqref{eq:equivalent},
    \begin{equation} \label{last}
        \binom{(2a+1)k+5}{ak+3}a^{ak+2}-\displaystyle\sum_{i=ak+2-(n-1)}^{ak+2}\binom{(2a+1)k+5}{i}a^i=\frac{P_{n}(a,k)}{Q_{n}(a,k)}\binom{(2a+1)k+5}{ak+3-n}a^{ak+2-n}.
    \end{equation}
From \eqref{ineq:Pnnegative}, one can check that $P_{n}(a,k)$ is negative when $k\leq a^{n-2}+(n-1)a^{n-3}+\frac{n(n-3)}{2}a^{n-4}$. Also, by the definition of $Q_{n}(a,k)$, we have $Q_{n}(a,k) >0$ when $\frac{n-3}{a}\leq k$. Thus, the ratio $\frac{P_{n}(a,k)}{Q_{n}(a,k)}$ is negative when $\frac{n-3}{a}\leq k\leq a^{n-2}+(n-1)a^{n-3}+\frac{n(n-3)}{2}a^{n-4}$ and thus the right-hand side of~\eqref{last} is negative. This implies~\eqref{secondtolast}.
\end{proof}

\subsection*{\textbf{2.}} 
Let $q$ be an indeterminate. For two polynomials $X(q), Y(q)\in\mathbb{R}_{\geq0}[q]$, we say $X(q)\leq Y(q)$ (or $Y(q) \geq X(q)$) if and only if $Y(q)-X(q)\in\mathbb{R}_{\geq0}[q]$. A sequence of polynomials with nonnegative coefficients $\{X_{i}(q)\}_{i\geq0}$ is $q$\textit{-log concave} if $X_{k}^{2}(q)\geq X_{k-1}(q)X_{k+1}(q)$ for all $k\geq1$, and it is \textit{strongly }$q$\textit{-log concave} if $X_{k}(q)X_{l}(q)\geq X_{k-1}(q)X_{l+1}(q)$ for all $0<k\leq l$. Clearly, strong $q$-log concavity implies $q$-log concavity, but unlike the ordinary case, these two notions are not equivalent. Natural $q$-analogues of several unimodal sequences, including $\{n\}_{n\in \mathbb{N}}$, $\{\binom{n}{i}\}_{i=0}^{n}$, are known to be strongly $q$-log concave (we refer to \cite{butler1990q} and \cite{sagan1992inductive} for a proof of these facts). Using a similar argument as in Lemma \ref{Lemma2.1}, one can show that for $a\in \RRgt$, the sequence $\{\frac{1}{(1+a)^r}\sum_{i=0}^{r}\binom{m}{i}_{q}a^i\}_{r=0}^{m}$ is strongly $q$-log concave, where $\binom{m}{i}_{q}$ are the Gaussian binomial coefficients. Furthermore, for $a\in \Zgt$, one can show that $\{\frac{1}{(1+a)^r}\sum_{i=0}^{r}\binom{m}{i}_{q}[a]_{q}^i\}_{r=0}^{m}$ has the same property, where $[a]_{q}\coloneqq\sum_{i=0}^{a-1}q^i$ is the $q$-integer. One natural question is what happens to these polynomials at $r=r_{a}$? Since this paper concerns the case when $q=1$, the main theorem of this paper implies that the sum of coefficients of the polynomials attains its maximum at $r=r_{a}$. Is there anything more interesting to be said about these two polynomial sequences?

\subsection*{\textbf{3.}} 
We proved that $\{\f{r}\}_{r=0}^{m}$ is a unimodal sequence for $a \in \mathbb{R}_{>0}$, but we described the peak position only for $a \in \mathbb{Z}_{\geq 1}$. Our computational experiments show that $\f{r}$ attains its maximal value close to $r=r_{a}=\lfloor\frac{am-(a-1)}{2a+1}\rfloor+1$ for $a \in \mathbb{R}_{>0}$ for almost all values of $m$. However, we were not able to formulate a conjecture for the location of the peak when $a \in \mathbb{R}_{>0}$. Also, we note that some parts of our proofs rely on $a$ being an integer, so even if one formulates a conjecture for the peak location for noninteger values of $a$, different techniques would be needed in the proof.

\subsection*{\textbf{4.}} In Section~\ref{sec:avgvalue}, we considered the probability distribution given by $p_{m,a}$. It might be interesting to see what the behavior of this distribution is for large values of $m$.


\section*{Acknowledgements} \label{sec:acknowledgments}

We want to thank Brian Fralix, Ryann Cartor, and Felice Manganiello for helpful discussions related to general questions that came up as we worked on this project. We also thank an anonymous reviewer for carefully reading the initial version of the paper and giving many insightful suggestions.

\bibliography{bibliography}{}

\begin{thebibliography}{10}

\bibitem{alizadeh2022sequences}
F.~Alizadeh, S.~Glasby, and C.~E. Praeger.
\newblock Sequences of linear codes where the rate times distance grows rapidly.
\newblock {\em Journal of Algebra Combinatorics Discrete Structures and Applications}, 2022.

\bibitem{robert1990ash}
R.~B. Ash.
\newblock {\em Information Theory}.
\newblock Dover Publications Inc., New York, 1990.

\bibitem{MR3409348}
P.~Br\"{a}nd\'{e}n.
\newblock Unimodality, log-concavity, real-rootedness and beyond.
\newblock In {\em Handbook of enumerative combinatorics}, Discrete Math. Appl. (Boca Raton), pages 437--483. CRC Press, Boca Raton, FL, 2015.

\bibitem{brenti1989unimodal}
F.~Brenti.
\newblock Unimodal, log-concave and {P}\'{o}lya frequency sequences in combinatorics.
\newblock {\em Mem. Amer. Math. Soc.}, 81(413):viii+106, 1989.

\bibitem{brenti1994update}
F.~Brenti.
\newblock Log-concave and unimodal sequences in algebra, combinatorics, and geometry: an update.
\newblock In {\em Jerusalem combinatorics '93}, volume 178 of {\em Contemp. Math.}, pages 71--89. Amer. Math. Soc., Providence, RI, 1994.

\bibitem{butler1990q}
L.~M. Butler.
\newblock The $q$-log-concavity of $q$-binomial coefficients.
\newblock {\em Journal of Combinatorial Theory, Series A}, 54(1):54--63, 1990.

\bibitem{glasby2022maximum}
S.~Glasby and G.~Paseman.
\newblock On the maximum of the weighted binomial sum $2^{-r}\sum_{i= 0}^ r \binom {m}{i}$.
\newblock {\em The Electronic Journal of Combinatorics}, pages P2--5, 2022.

\bibitem{ling2004coding}
S.~Ling and C.~Xing.
\newblock {\em Coding theory: a first course}.
\newblock Cambridge University Press, 2004.

\bibitem{sagan1992inductive}
B.~E. Sagan.
\newblock Inductive proofs of $q$-log concavity.
\newblock {\em Discrete mathematics}, 99(1-3):289--306, 1992.

\bibitem{stanley1989log}
R.~P. Stanley.
\newblock Log-concave and unimodal sequences in algebra, combinatorics, and geometry.
\newblock {\em Ann. New York Acad. Sci}, 576(1):500--535, 1989.

\end{thebibliography}
\bibliographystyle{abbrv}

\end{document}